\documentclass[11pt,reqno]{amsart}
\calclayout
\usepackage{amsmath}
\usepackage[colorlinks=true,linkcolor=blue,citecolor=blue]{hyperref}
\usepackage{amssymb,amsthm,xypic,stmaryrd,graphicx,mathtools}
\usepackage[all]{xy}
\usepackage{bbm, dsfont}
\usepackage{boxedminipage,xcolor}
\usepackage[shortlabels]{enumitem}
\usepackage{thmtools}
\usepackage{thm-restate}
\usepackage{tikz-cd}
\usepackage{empheq}
\usetikzlibrary{arrows.meta}
\tikzcdset{arrow style=tikz}
\numberwithin{equation}{section}
\DeclareMathOperator{\supp}{supp}
\DeclareMathOperator{\propagation}{prop}

\DeclareMathOperator{\Ind}{Ind}

\DeclareMathOperator{\tc}{tc}

\DeclareMathOperator{\ind}{index}

\DeclareMathOperator{\Spin}{Spin}

\newcommand{\beq}[1]{\begin{equation} \label{#1}}
\newcommand{\eeq}{\end{equation}}

\newcommand{\bea}{\begin{eqnarray}}
\newcommand{\eea}{\end{eqnarray}}

\begin{document}

\theoremstyle{plain}
\newtheorem{theorem}{Theorem}[section]
\newtheorem{thm}{Theorem}[section]
\newtheorem{lemma}[theorem]{Lemma}
\newtheorem{proposition}[theorem]{Proposition}
\newtheorem{prop}[theorem]{Proposition}
\newtheorem{corollary}[theorem]{Corollary}
\newtheorem{conjecture}[theorem]{Conjecture}
\newtheorem{question}[theorem]{Question}

\theoremstyle{definition}
\newtheorem{convention}[theorem]{Convention}
\newtheorem{definition}[theorem]{Definition}
\newtheorem{defn}[theorem]{Definition}
\newtheorem{example}[theorem]{Example}
\newtheorem{remark}[theorem]{Remark}
\newtheorem*{remark*}{Remark}
\newtheorem*{overview*}{Overview}
\newtheorem*{results*}{Results}
\newtheorem{rem}[theorem]{Remark}

\newcommand{\C}{\mathbb{C}}
\newcommand{\R}{\mathbb{R}}
\newcommand{\Z}{\mathbb{Z}}
\newcommand{\N}{\mathbb{N}}
\newcommand{\Q}{\mathbb{Q}}

\newcommand{\Supp}{{\rm Supp}}

\newcommand{\field}[1]{\mathbb{#1}}
\newcommand{\bZ}{\field{Z}}
\newcommand{\bR}{\field{R}}
\newcommand{\bC}{\field{C}}
\newcommand{\bN}{\field{N}}
\newcommand{\bT}{\field{T}}
\newcommand{\cB}{{\mathcal{B} }}
\newcommand{\cK}{{\mathcal{K} }}
\newcommand{\cF}{{\mathcal{F} }}
\newcommand{\cO}{{\mathcal{O} }}
\newcommand{\cE}{\mathcal{E}}
\newcommand{\cS}{\mathcal{S}}
\newcommand{\cN}{\mathcal{N}}
\newcommand{\calL}{\mathcal{L}}

\newcommand{\KK}{K \! K}

\newcommand{\norm}[1]{\| #1\|}

\newcommand{\Spinc}{\Spin^c}

\newcommand{\HH}{{\mathcal{H} }}
\newcommand{\Hpi}{\HH_{\pi}}

\newcommand{\DNR}{D_{N \times \R}}


\def\kt{\mathfrak{t}}
\def\kk{\mathfrak{k}}
\def\kp{\mathfrak{p}}
\def\kg{\mathfrak{g}}
\def\kh{\mathfrak{h}}
\def\so{\mathfrak{so}}
\def\cut{c}

\newcommand{\ddt}{\left. \frac{d}{dt}\right|_{t=0}}

\newcommand{\Todo}{\textbf{To do}}

\title{Functoriality for Higher Rho Invariants of Elliptic Operators}

\author{Hao Guo}
\address[Hao Guo]{ Department of Mathematics, Texas A\&M University }
\email{haoguo@math.tamu.edu}

\author{Zhizhang Xie}
\address[Zhizhang Xie]{ Department of Mathematics, Texas A\&M University }
\email{xie@math.tamu.edu}

\author{Guoliang Yu}
\address[Guoliang Yu]{ Department of
	Mathematics, Texas A\&M University}
\email{guoliangyu@math.tamu.edu}
\thanks{H.G. is partially supported by NSF 1564398. Z.X. is partially supported by NSF 1500823 and 1800737. G.Y. is partially supported by NSF 1700021, 1564398, and the Simons Fellows Program}


\maketitle

\begin{abstract}
Let $N$ be a closed spin manifold with positive scalar curvature and $D_N$ the Dirac operator on $N$. Let $M_1$ and $M_2$ be two Galois covers of $N$ such that $M_2$ is a quotient of $M_1$. Then the quotient map from $M_1$ to $M_2$ naturally induces maps between the geometric $C^*$-algebras associated to the two manifolds. We prove, by a finite-propagation argument, that the \emph{maximal} higher rho invariants of the lifts of $D_N$ to $M_1$ and $M_2$ behave functorially with respect to the above quotient map. This can be applied to the computation of higher rho invariants, along with other related invariants.
\end{abstract}
\vspace{1cm}
\tableofcontents

\section{Introduction}
\label{sec intro}
 An elliptic differential operator on a closed manifold has a Fredholm index. When such an operator is lifted to a covering space, one can, by taking into account the group of symmetries, define a far-reaching generalization of the Fredholm index, called the \emph{higher index} \cite{Baum-Connes,BCH,ConnesNCG,Kasparov,WillettYu}.
The higher index serves as an obstruction to the existence of invariant metrics of positive scalar curvature. In the case that such a metric exists, so that the higher index of the lifted operator vanishes, a \emph{secondary invariant} called the \emph{higher rho invariant} \cite{Roetopology,HR3} can be defined. The higher rho invariant is an obstruction to the inverse of the operator being local \cite{Hongzhi}. For some recent applications of the higher index and higher rho invariant to problems in geometry and topology, we refer the reader to \cite{CWXY,Song-Tang,Wang-Wang,WXY,XY1,XY2,XYZ}.

The main purpose of this paper is to prove that the higher rho invariant behaves functorially under appropriate maps between covering spaces. 

More precisely, suppose $M_1$ and $M_2$ are two Galois covers of a closed spin manifold $N$ with deck transformation groups $\Gamma_1$ and $\Gamma_2\cong \Gamma_1/H$, where $H$ is a normal subgroup of $\Gamma_1$. 
The natural projection $\pi\colon M_1\rightarrow M_2$ induces a family of maps between various \emph{geometric $C^*$-algebras} associated to $M_1$ and $M_2$, called \emph{folding maps}, which will be reviewed in section \ref{sec prelim}.
In particular, there is a folding map
\begin{align*}
\Psi\colon C^*_{\textnormal{max}}(M_1)^{\Gamma_1}&\rightarrow C^*_{\textnormal{max}}(M_2)^{\Gamma_2}
\end{align*}
between maximal equivariant Roe algebras on $M_1$ and $M_2$ with the property that it preserves small propagation of operators. Consequently, $\Psi$ induces a map 
$$\Psi_{L,0}\colon C^*_{L,0,\textnormal{max}} (M_1)^{\Gamma_1} \rightarrow C^*_{L,0,\textnormal{max}} (M_2)^{\Gamma_2}$$
at the level of obstruction algebras. The main result of this paper is that the $K$-theoretic map induced by $\Psi_{L,0}$ relates the maximal higher rho invariants of Dirac operators on $M_1$ and $M_2$:
\begin{theorem}
	\label{thm funct higher rho}
	Let $N$ be a closed, spin Riemannian manifold with positive scalar curvature. Let $D_N$ be the Dirac operator on $N$. Let $M_1$ and $M_2$ be Galois covers of $M$ with deck transformation groups $\Gamma_1$ and $\Gamma_2\cong\Gamma_1/H$ respectively, for a normal subgroup $H$ of $\Gamma_1$. Let $D_1$ and $D_2$ be the lifts of the Dirac operator on $N$ to $M_1$ and $M_2$ respectively. Then
	$$(\Psi_{L,0})_*(\rho_{\textnormal{max}}(D_1))=\rho_{\textnormal{max}}(D_2)\in K_\bullet(C^*_{L,0,\textnormal{max}}(M_2)^{\Gamma_2}),$$
	where $(\Psi_{L,0})_*$ is the map on $K$-theory induced by $\Psi_{L,0}$ and $\rho_{\textnormal{max}}(D_j)$ is the maximal higher rho invariant of $D_j$.
\end{theorem}
We also generalize this result to the non-cocompact setting using the same method of proof -- see section \ref{sec generalizations} of this paper.

Theorem \ref{thm funct higher rho} is useful for computing higher rho and related invariants. It has been applied in the recent work of Wang, Xie, and Yu \cite{Jinmin2} to compute the delocalized eta invariant on a covering space of a closed manifold by showing that, under suitable geometric conditions, it can be approximated by delocalized eta invariants on finite-sheeted covers, which are more computable. We mention two of their results.

Let $M$ be a closed spin manifold equipped with a positive scalar curvature metric. First, the authors proved that given a finitely generated discrete group $\Gamma$ and a sequence $\{\Gamma_i\}_{i\in\mathbb{N}}$ of finite-index normal subgroups of $\Gamma$ that \emph{distinguishes}, in a suitable sense (\cite[Definition 2.3]{WXY}), a given non-trivial conjugacy class in $\Gamma$, the sequence of associated delocalized eta invariants on the $\Gamma_i$-Galois covers of $M$ stabilizes, under the assumption that the maximal Baum-Connes assembly map for $\Gamma$ is a rational isomorphism. We refer to \cite[Theorem 1.1]{WXY} for more details.

Second, the authors proved that if $\widetilde{D}$ is the Dirac operator associated to the $\Gamma$-Galois cover of $M$, then if the spectral gap of $\widetilde{D}$ at zero is sufficiently large, the delocalized eta invariant of $\widetilde{D}$ is equal to the limit of those of Dirac operators associated to the $\Gamma_i$-Galois covers of $M$ \cite[Theorem 1.4]{WXY}. In particular, this is true if the group $\Gamma$ has subexponential growth \cite[Corollary 1.5]{WXY}.


Finally, we mention that in a new preprint \cite{Liu-Tang-Xie-Yao-Yu} by Liu, Tang, Xie, Yao, and Yu, the authors apply Theorem \ref{thm funct higher rho} to study rigidity of \emph{relative eta invariants}. These invariants are obtained by pairing the higher rho invariant with certain traces, similar to the way in which the delocalized eta invariant is related to the higher rho invariant through a trace \cite{XY2}.

%


\hfill\vskip 0.05in
\begin{overview*}
\hfill\vskip 0.05in
\noindent The paper is organized as follows. We begin in section \ref{sec prelim} by recalling the geometric and operator-algebraic setup we work with. In section \ref{sec folding} we define the folding maps and establish some of their basic properties. In section \ref{sec funct calc} we develop the analytical properties of the wave operator in the maximal setting. These tools are then put to use in section \ref{sec higher ind}, where we give a new proof of the functoriality property of the maximal equivariant higher index. This serves as an intermediate step towards Theorem \ref{thm funct higher rho}, which is proved in section \ref{sec higher rho}. In section \ref{sec generalizations} we provide generalizations of our results to the non-cocompact setting.
\end{overview*}
\subsection*{Acknowledgements}
\hfill\vskip 0.05in
\noindent The authors are grateful to Peter Hochs for pointing us to the reference \cite[Theorem 1.1]{MislinValette} of Alain Valette for functoriality of the maximal higher index.

\hfill\vskip 0.2in
\section{Preliminaries}
\label{sec prelim}
In this section, we fix some notation before introducing the necessary operator-algebraic background and geometric setup for our results.
\subsection{Notation}
\label{subsec notation}
\hfill\vskip 0.05in
	\noindent For $X$ a Riemannian manifold, we write $B(X)$, $C_b(X)$, $C_0(X)$, and $C_c(X)$ to denote the $C^*$-algebras of complex-valued functions on $X$ that are, respectively: bounded Borel, bounded continuous, continuous and vanishing at infinity, and continuous with compact support. A superscript `$^\infty$' may be added where appropriate to indicate the additional requirement of smoothness. 
	
	We write $d_X$ for the Riemannian distance function on $X$ and $\mathbbm{1}_S$ for the characteristic function of a subset $S\subseteq X$.
	
	For any $C^*$-algebra $A$, we denote its unitization by $A^+$, its multiplier algebra by $\mathcal{M}(A)$, and view $A$ as an ideal of $\mathcal{M}(A)$. 
	
	
	The action of a group $G$ on $X$ naturally induces a $G$-action on spaces of functions on $X$ as follows: given a function $f$ on $X$ and $g\in G$, define $g\cdot f$ by $g\cdot f(x)=f(g^{-1}x)$. More generally, for a section $s$ of a $\Gamma$-vector bundle over $X$, the section $g\cdot s$ is defined by $g\cdot s(x)=g(s(g^{-1}x))$. We say that an operator on sections of a bundle is $G$-equivariant if it commutes with the $G$-action.
	
	
\hfill\vskip 0.1in
\subsection{Geometric $C^*$-algebras}
	\label{subsec op alg}
 	\hfill\vskip 0.05in
	\noindent We now recall the notions of geometric modules and their associated $C^*$-algebras. Throughout this subsection, $X$ is a Riemannian manifold equipped with a proper isometric action by a discrete group $G$.
		\begin{definition}
		\label{def XGmodule}
	An $X$-$G$-module is a separable Hilbert space $\mathcal{H}$ equipped with a non-degenerate $*$-representation $\rho\colon C_0(X)\rightarrow\mathcal{B}(\mathcal{H})$ and a unitary representation $U\colon G\rightarrow\mathcal{U}(\mathcal{H})$ such that for all $f\in C_0(X)$ and $g\in G$, we have  $U_g\rho(f)U_g^*=\rho(g\cdot f)$.
	\end{definition}
	For brevity, we will omit $\rho$ from the notation when it is clear from context.
	\begin{definition}
	\label{def:suppprop}
		Let $\mathcal{H}$ be an $X$-$ G$-module and $T\in\mathcal{B}(\mathcal{H})$.
		\begin{itemize}
			\item The \emph{support} of $T$, denoted $\textnormal{supp}(T)$, is the complement of all $(x,y)\in X\times X$ for which there exist $f_1,f_2\in C_0(X)$ such that $f_1(x)\neq 0$, $f_2(y)\neq 0$, and
			$$f_1Tf_2=0;$$
			\item The \textit{propagation} of $T$ is the extended real number $$\textnormal{prop}(T)=\sup\{d_X(x,y)\,|\,(x,y)\in\textnormal{supp}(T)\};$$
			\item $T$ is \textit{locally compact} if $fT$ and $Tf\in\mathcal{K}(\mathcal{H})$ for all $f\in C_0(X)$;
			\item $T$ is \textit{$ G$-equivariant} if $U_g TU_g^*=T$ for all $g\in G$;

		\end{itemize}
		The \emph{equivariant algebraic Roe algebra for $\mathcal{H}$}, denoted $\mathbb{C}[X;\mathcal{H}]^ G$, is the $*$-subalgebra of $\mathcal{B}(\mathcal{H})$ consisting of $ G$-equivariant, locally compact operators with finite propagation. 
		
	\end{definition}

	We will work with the maximal completion of the equivariant algebraic Roe algebra. To ensure that this completion is well-defined, we require that the module $\mathcal{H}$ satisfy an additional admissibility condition. To define what this means, we need the following fact: if $H$ is a Hilbert space and $\rho\colon C_0(X)\rightarrow\mathcal{B}(H)$ is a non-degenerate $*$-representation, then $\rho$ extends uniquely to a $*$-representation $\widetilde{\rho}\colon B(X)\rightarrow\mathcal{B}(H)$ subject to the property that, for a uniformly bounded sequence in $B(X)$ converging pointwise, the corresponding sequence in $\mathcal{B}(H)$ converges in the strong topology.

	\begin{definition}[\cite{Yu}]
		\label{def admissible}
		Let $\mathcal{H}$ be an $X$-$G$-module as in Definition \ref{def XGmodule}. We say that $\mathcal{H}$ is \emph{admissible} if:
		\begin{enumerate}[(i)]
		\item For any non-zero $f\in C_0(X)$ we have $\rho(f)\notin\mathcal{K}(\mathcal{H})$;
		\item For any finite subgroup $F$ of $ G$ and any $F$-invariant Borel subset $E\subseteq X$, there is a Hilbert space $H'$ equipped with the trivial $F$-representation such that $\widetilde{\rho}(\mathbbm{1}_E)H'\cong l^2(F)\otimes H'$ as $F$-representations, where $\widetilde{\rho}$ is defined by extending $\rho$ as above.
		\end{enumerate}
	\end{definition}
	If an $X$-$ G$-module $\mathcal{H}$ is admissible, we will write $\mathbb{C}[X]^ G$ in place of $\mathbb{C}[X;\mathcal{H}]^ G$,
	for the reason that $\mathbb{C}[X;\mathcal{H}]^G$ is independent of the choice of admissible module -- see \cite[Chapter 5]{WillettYu}.
	
	\begin{remark}
		When $ G$ acts freely and properly on $X$, the Hilbert space $L^2(X)$ is an admissible $X$-$ G$-module. In the case that the action is not free, $L^2(X)$ can always be embedded into a larger admissible module.
	\end{remark}
	

		\begin{definition}
			\label{def:maximalnorm}
			The \emph{maximal norm} of an operator $T\in\mathbb{C}[X]^{ G}$ is
			$$||T||_{\textnormal{max}}\coloneqq\sup_{\phi,H'}\left\{\norm{\phi(T)}_{\mathcal{B}(H')}\,|\,\phi\colon\mathbb{C}[X]^{ G}\rightarrow\mathcal{B}(H')\textnormal{ is a $*$-representation}\right\}.$$
			The \emph{maximal equivariant Roe algebra of $M_j$}, denoted $C^*_{\text{max}}(X)^ G$, is the completion of $\mathbb{C}[X]^ G$ in the norm $||\cdot||_{\textnormal{max}}$.
		\end{definition}
		\begin{remark}
			To make sense of Definition \ref{def:maximalnorm} for general $X$ and $G$, one first needs to establish finiteness of the quantity $\norm{\cdot}_{\textnormal{max}}$. It was shown in \cite{GWY} that if $G$ acts on $X$ freely and properly with compact quotient, the norm $\norm{\cdot}_{\textnormal{max}}$ is finite. This was generalized in \cite{GXY} to the case when $X$ has bounded geometry and the $ G$-action satisfies a suitable geometric assumption.
		\end{remark}

		\begin{remark}
		\label{rem kernel algebra}
			Equivalently, one can obtain $C^*_{\text{max}}(X)^ G$ by taking the analogous maximal completion of the subalgebra $\mathcal{S}^ G$ of $\mathbb{C}[X]^{ G}$ consisting of those operators given by smooth Schwartz kernels.
		\end{remark}
		
	\begin{definition}
	\label{def localization}
		Consider the $*$-algebra $L$ of functions $f\colon [0,\infty)\rightarrow C^*_{\textnormal{max}}(X)^ G$ that are uniformly bounded, uniformly continuous, and such that
		$$\textnormal{prop}(f(t))\rightarrow 0\quad\textnormal{as }t\rightarrow\infty.$$
		\begin{enumerate}[(i)]
			\item The \emph{maximal equivariant localization algebra}, denoted by $C^*_{L,\textnormal{max}}(X)^ G$, is the $C^*$-algebra obtained by completing $L$ with respect to the norm $$\norm{f}\coloneqq\sup_t\norm{f(t)}_{\textnormal{max}};$$
			\item The map $L\rightarrow \mathbb{C}[X]^ G$ given by $f\mapsto f(0)$ extends to the \emph{evaluation map} $$\textnormal{ev}\colon C^*_{L,\textnormal{max}}(X)^ G\rightarrow C^*_{\textnormal{max}}(X)^ G;$$
			\item The \emph{maximal equivariant obstruction algebra} is $C^*_{L,0,\textnormal{max}}(X)^ G\coloneqq\ker(\textnormal{ev})$.
		\end{enumerate}
	\end{definition}
\hfill\vskip 0.1in
\subsection{Geometric setup}
\hfill\vskip 0.05in
\label{subsec geom setup}
\noindent We will work with the following geometric setup. 

Let $(N,g_N)$ be a closed Riemannian manifold. Let $D_N$ be a first-order essentially self-adjoint elliptic differential operator on a bundle $E_N\rightarrow N$. We will assume throughout that if $N$ is odd-dimensional then $D_N$ is an ungraded operator, while if $N$ is even-dimensional then $D_N$ is odd-graded with respect to a $\mathbb{Z}_2$-grading on $E_N$.

Let $p_1\colon M_1\rightarrow N$ and $p_2\colon M_2\rightarrow N$ be two Galois covers of $N$ with deck transformation groups $\Gamma_1$ and $\Gamma_2$ respectively. We will assume throughout this paper that $\Gamma_2\cong\Gamma_1/H$ for some normal subgroup $H$ of $\Gamma_1$, so that $M_2\cong M_1/H$. 
Let $\pi\colon M_1\rightarrow M_2$ be the projection map. Note that for $j=1,2$, the group $\Gamma_j$ acts freely and properly on $M_j$.

For $j=1$ or $2$, let $g_j$ be the lift of the Riemannian metric $g_N$ to $M_j$. Let $E_j$ be the pullback of $E$ along the covering map $M_j\rightarrow M$, equipped with the natural $\Gamma_j$-action. Since $D$ acts locally, it lifts to a $\Gamma_j$-equivariant operator $D_j$ on $C^\infty(E_j)$.

We will apply the notions in subsection \ref{subsec op alg} with $X=M_j$ and $G=\Gamma_j$. The Hilbert space $L^2(E_j)$, equipped with the natural $\Gamma_j$-action on sections and the $C_0(M_j)$-representation defined by pointwise multiplication, is an $M_j$-$\Gamma_j$-module in the sense of Definition \ref{def:suppprop}.

Moreover, the fact that the $\Gamma_j$-action on $M_j$ is free and proper implies that $L^2(E_j)$ is admissible in the sense of Definition \ref{def admissible}. To see this, choose a compact, Borel fundamental domain $\mathcal{D}_j$ for the $\Gamma_j$-action on $M_j$
such that for each $\gamma\in\Gamma_j$, the restriction
$$p_j|_{\gamma\cdot\mathcal{D}_j}\colon \gamma\cdot\mathcal{D}_j\rightarrow N$$
is a Borel isomorphism, where the projection maps $p_j$ are as in subsection \ref{subsec geom setup}. For each $j=1,2$ we have a map
\begin{align}
\label{eq Phi}
\Phi_j\colon L^2(E_j)&\to l^2(\Gamma_j)\otimes L^2(E_j|_{\mathcal{D}_j}),\nonumber\\
	s&\mapsto\sum_{\gamma\in\Gamma_j}\gamma\otimes\gamma^{-1}\chi_{\gamma\mathcal{D}_j}s.
\end{align}
This is a $\Gamma_j$-equivariant unitary isomorphism with respect to the tensor product of the left-regular representation on $l^2(\Gamma_j)$ and the trivial representation on $L^2(E_j|_{\mathcal{D}_j})$. Conjugation by $\Phi_j$ induces a $*$-isomorphism
\begin{equation}
\label{eq equiv Roe iso}
	\mathbb{C}[M_j]^{\Gamma_j}\cong\mathbb{C}\Gamma_j\otimes\mathcal{K}(L^2(E_j|_{\mathcal{D}_j})).
\end{equation}
\begin{remark}
\label{rem max norm}
It follows from \eqref{eq equiv Roe iso} that for any $r>0$, there exists a constant $C_r$ such that for all $T\in\mathbb{C}[M_j]^{\Gamma_j}$ with $\textnormal{prop}(T)\leq r$, we have
$$\norm{T}_{\textnormal{max}}\leq C_r\norm{T}_{\mathcal{B}(L^2(E_j))}.$$
Thus the maximal equivariant Roe algebra $C^*_{\textnormal{max}}(M_j)^{\Gamma_j}$ from Definition \ref{def:maximalnorm} is well-defined in our setting. Moreover,
$$C^*_{\textnormal{max}}(M_j)^{\Gamma_j}\cong C^*_{\textnormal{max}}(\Gamma_j)\otimes\mathcal{K},$$
so that the $K$-theories of both sides are isomorphic.
\end{remark}


\hfill\vskip 0.2in
\section{Folding maps}
\label{sec folding}
In this section, we define certain natural $*$-homomorphisms, called \emph{folding maps}, between geometric $C^*$-algebras of covering spaces, and discuss the role they play in functoriality for higher invariants. These maps were introduced in \cite[Lemma 2.12]{CWY}.

To begin, let us provide some motivation at the level of groups. Observe that the quotient homomorphism $\Gamma_1\rightarrow\Gamma_2\cong\Gamma_1/H$ induces a natural surjective $*$-homomorphism between group algebras,
$$\alpha\colon \mathbb{C}\Gamma_1\rightarrow\mathbb{C}\Gamma_2,\quad\sum_{i=1}^k c_{\gamma_i}\gamma_i\mapsto\sum_{i=1}^k c_{\gamma_i}[\gamma_i],$$
where $[\gamma]$ is the class of an element $\gamma\in\Gamma_1$ in $\Gamma_1/H$. If one views elements of $\mathbb{C}\Gamma_j$ as kernel operators on the Hilbert space $l^2(\Gamma_j)$, then the map $\alpha$ takes a kernel $k\colon\Gamma_1\times\Gamma_1\rightarrow\mathbb{C}$ to the kernel
$$\alpha(k)\colon \Gamma_2\times\Gamma_2\rightarrow\mathbb{C},\quad\big([\gamma],[\gamma']\big)\mapsto\sum_{h\in H}k(h\gamma,\gamma').$$

There is an analogous map between kernels at the level of the Galois covers $M_1$ and $M_2\cong M_1/H$. Given a smooth, $\Gamma_1$-equivariant Schwartz kernel $k(x,y)$ with finite propagation on $M_1$, one can define a smooth, $\Gamma_2$-equivariant Schwartz kernel $\psi(k)$ with finite propagation on $M_2$ by the formula
\begin{equation}
\label{eq Psi k}
\psi(k)([x],[y])=\sum_{h\in H}k(hx,y).
\end{equation}
Note that this sum is finite by properness of the $H$-action on $M_1$. The formula \eqref{eq Psi k} defines a map
$$\psi\colon\mathcal{S}^{\Gamma_1}\to\mathcal{S}^{\Gamma_2},$$
where the kernel algebras $\mathcal{S}^{\Gamma_1}$ and $\mathcal{S}^{\Gamma_2}$ are as in Remark \ref{rem kernel algebra}. This map is a $*$-homomorphism by \cite[Lemma 5.2]{GHM2}.

 \hfill\vskip 0.1in
\subsection{Definition of the folding map}
 	\hfill\vskip 0.05in
\noindent We now define a more general version of the map \eqref{eq Psi k}, called the \emph{folding map} $\Psi$, at the level of finite-propagation operators.

Let $\mathcal{B}_{\textnormal{fp}}(L^2(E_j))^{\Gamma_j}$ denote the $*$-algebra of bounded, $\Gamma_j$-equivariant operators on $L^2(E_j)$ with finite propagation. The folding map $\Psi$ is a $*$-homomorphism $\mathcal{B}_{\textnormal{fp}}(L^2(E_1))^{\Gamma_1}\to\mathcal{B}_{\textnormal{fp}}(L^2(E_2))^{\Gamma_2}$ with the following properties:
\begin{itemize}
	\item For any $T\in\mathcal{B}_{\textnormal{fp}}(L^2(E_1))^{\Gamma_1}$, we have $\textnormal{prop}\big(\Psi(T)\big)\leq\textnormal{prop}(T)$;
	\item $\Psi$ is surjective;
	\item $\Psi$ restricts to a surjective $*$-homomorphism $\mathbb{C}[M_1]^{\Gamma_1}\rightarrow\mathbb{C}[M_2]^{\Gamma_2}$.
	\end{itemize}
These properties are proved in Propositions \ref{prop folding}, \ref{prop prop}, and \ref{prop psi surj} below. 

To define $\Psi$, it will be convenient to use partitions of unity on $M_1$ and $M_2$ that are compatible with the $\Gamma_1$ and $\Gamma_2$-actions, as follows.
%
Since $N$ is compact, there exists $\epsilon>0$ such that for each $x\in N$, the ball $B_\epsilon(x)$ is evenly covered with respect to both $p_1\colon M_1\to N$ and $p_2\colon M_2\to N$. Let
\begin{equation}
\label{eq POU N}
	\mathcal{U}_{N}\coloneqq\{U_i\}_{i\in I}
\end{equation}
be a finite open cover of $N$ such that each $U_i$ has diameter at most $\epsilon$. Let $\{\phi_i\}_{i\in I}$ be a partition of unity subordinate to $\mathcal{U}_N$. 

For each $i\in I$, let $U_i^e$ be a lift of $U_i$ to $M_1$ via the covering map $p_1\colon M_1\to N$. Similarly, let $\phi_i^e$ be the lift of $\phi_i$ to $U_i^e$. For each $g\in\Gamma_1$, let $U_i^g$ and $\phi_i^g$ be the $g$-translates of $U_i^e$ and $\phi_i^e$ in $M_1$ respectively. Then
\begin{equation}
\label{eq POU M1}
\mathcal{U}_{M_1}\coloneqq\big\{U_i^g\big\}_{\substack{i\in I,\,g\in\Gamma_1}},\qquad\{\phi_i^g\}_{i\in I,\,g\in\Gamma_1}
\end{equation}
define a locally finite, $\Gamma_1$-invariant open cover of $M_1$ and a subordinate partition of unity. This partition of unity is $\Gamma_1$-equivariant in the sense that for each $i\in I$, $g\in\Gamma$, and $x\in\textnormal{supp}(\phi_i^e)$, we have $\phi_i^g(gx)=\phi_i^e(x)$.

After taking a quotient by the $H$-action, we obtain a $\Gamma_2$-invariant open cover of $M_2$, together with a subordinate $\Gamma_2$-equivariant partition of unity:
\begin{equation}
\label{eq POU M2}
\mathcal{U}_{M_2}\coloneqq\big\{U_i^{[g]}\big\}_{\substack{i\in I,\,[g]\in\Gamma_2}},\qquad\big\{\phi_i^{[g]}\big\}_{i\in I,\,[g]\in\Gamma_2}.
\end{equation}
For each $i$ and $g$, we have the following commutative diagram of isomorphisms of local structures involving the covering map $\pi$:
\begin{center}
\begin{tikzcd}
E_1|_{U_i^g} \arrow[r] \arrow[d]
& U_i^{g}\arrow[d, "\pi"] \\
E_2|_{U_i^{[g]}} \arrow[r]
& \,U_i^{[g]}.
\end{tikzcd}
\end{center}
We will write $\pi_*$ and $\pi|_{U_i^g}^*$ for the maps relating a local section of $E_1$ to the corresponding local section of $E_2$. That is, if $u$ is a section of $E_1$ over $U_i^g$, and $v$ is the corresponding section of $E_2$ over $U_i^{[g]}$, then we have
\begin{center}
\begin{tikzcd}[every arrow/.append style={shift left}]
u \arrow[r, maps to, "\pi_*"] & v \arrow[l, maps to, "\pi|_{U_i^g}^*"]\end{tikzcd}
\end{center}
If $w$ is any vector in the bundle $E_1$, we will also write $\pi_*(w)$ for its image under the projection $E_1\to E_2$.

We are now ready for the definition of the folding map $\Psi$. Choose a set $S\subseteq\Gamma_1$ of coset representatives for $\Gamma_2\cong H\backslash\Gamma_1$. Given an operator $T\in \mathcal{B}_{\textnormal{fp}}(L^2(E_1))^{\Gamma_1}$, we define $\Psi(T)$ to be the operator on $L^2(E_2)$ that takes a section $u\in L^2(E_2)$ to the section
	\begin{align}
	\label{eq folding complicated}
	\Psi(T)u\coloneqq\sum_{\substack{g\in\Gamma_1,\\s\in S}}\sum_{i,j\in I}\pi_*\big(\phi_i^g\cdot T\circ\pi|_{U_j^{s}}^*\circ\phi_j^{[s]}(u)\big)\in L^2(E_2).
	\end{align}
	To clarify the presentation of this equation and others like it, we adopt the following notational convention for moving between equivalent local sections of $E_1$ and $E_2$:
	\begin{convention} For $v\in L^2(E_1)$ and $u\in L^2(E_2)$, we will use the short-hand
	\label{conv convention}
	\begin{itemize}
	\item $\phi_i^{[g]}v$ to denote $\pi_*(\phi_i^{g}v)$;
	\item $\phi_i^gu$ to denote $\pi|^*_{U_i^g}(\phi_i^{[g]}u)$.
	\end{itemize}
	(Note that the meaning of $\phi_i^{[g]}v$ depends on the representative $g$.) 
	\end{convention}

Using this convention, the formula \eqref{eq folding complicated} reads
	\begin{align}
	\label{eq folding}
	\Psi(T)u\coloneqq\sum_{\substack{g\in\Gamma_1,\\s\in S}}\sum_{i,j\in I}\phi_i^{[g]}T(\phi_j^{s}u)\in L^2(E_2).
	\end{align}
	We call $\Psi$ the \emph{folding map}. The fact that the above definition is independent of the choices of the set $S$ and compatible partitions of unity is proved in Proposition \ref{prop folding}.

	Before proceeding further, let us record a few straightfoward identities that will help us navigate through notational clutter:
\begin{enumerate}[(i)]
\item If $v$ is a local section of $E_1$ supported in some neighborhood $U_j^{g}$, then for any $g'\in\Gamma_1$ we have
\begin{align}
g'\big(\phi_i^{g}v\big)&=\phi_i^{g'g}g'v,\label{eq identity2}\\
\pi_*(g'v)&=[g']\cdot\pi_*(v).\label{eq identity4}
\end{align}
\item If $u$ is a local section of $E_2$ supported in some neighborhood $U_j^{[g]}$, then for any $g,g'\in\Gamma_1$ and $j\in I$ we have
\begin{align}
g'\big(\pi|_{U_j^{g}}^*u\big)&=\pi|_{U_j^{g'g}}^*([g']\cdot u),\label{eq identity1}\\
u&=\sum_{\substack{i\in I,\\\gamma\in\Gamma_1}}\phi_i^{[\gamma]}(\pi|_{U_j^{g}}^*u).\label{eq identity3}
\end{align}
\end{enumerate}

	The following lemma provides a convenient pointwise formula for $\Psi(T)u$:
\begin{lemma} For any $u\in L^2(E_2)$ and $x\in M_2$, we have
	\begin{equation}
	\label{eq pointwise folding}
	\big(\Psi(T)u)(x)=\pi_*\Big(\sum_{\substack{j\in I,\\g\in\Gamma_1}}T(\phi_j^{g}u)(y_0)\Big),
	\end{equation}
	where $y_0$ is any point in $\pi^{-1}(x)$, and the sum on the right-hand side is finite.
	\end{lemma}
\begin{proof}
	Observe that for any $j\in I$ and $g\in\Gamma_1$, the summand $T(\phi_j^{g}u)(y_0)$ may only be non-zero if $U_j^g\cap B_{\textnormal{prop}\,T}(y_0)\neq\emptyset$. Since $T$ has finite propagation and the $\Gamma_1$-action is proper, the set
	$$\{(g,j)\,|\,U_j^g\cap B_{\textnormal{prop}\,T}(y_0)\neq\emptyset\}$$
	is finite. Hence the sum in question is finite for every $x\in M_2$ and $y_0\in\pi^{-1}(x)$. 
	
	To prove \eqref{eq pointwise folding}, let us first rewrite the right-hand side as
	\begin{equation}
	\label{eq rewrite}
	\pi_*\Big(\sum_{\substack{j\in I}}\sum_{\substack{h\in H,\\s\in S}}T(\phi_j^{hs}u)(y_0)\Big).
	\end{equation}

	Next, using \eqref{eq identity1}, \eqref{eq identity3}, and $\Gamma_1$-equivariance of $T$, each summand in \eqref{eq rewrite} equals
	\begin{align}
	\label{eq pointwise calculation}
	T(\phi_j^{hs}u)(y_0)&=T\big(h(\phi_j^{s}([h]\cdot u))\big)(y_0)\nonumber\\
	&=T\big(h(\phi_j^{s}u)\big)(y_0)\nonumber\\
	&=h\big(T(\phi_j^{s}u)\big)(y_0)\nonumber\\
	&=\Big(h\cdot\sum_{\substack{i\in I,\\g\in\Gamma_1}}\phi_i^g T\phi_j^{s}u\Big)(y_0),
	\end{align}
	where we have used Convention \ref{conv convention} when writing $\phi_j^s u$. Now observe that if $v$ is a section of $E_1$ supported in some open set $U_i^{g}$, then for any $x\in M_2$, we have
	\begin{equation}
\label{eq identity5}
(\pi_*v)(x)=\pi_*\big(\sum_{h\in H}(hv)(y_0)\big),
\end{equation}
	where $y_0$ is an arbitrary point in the inverse image $\pi^{-1}(x)$. Keeping this in mind while summing \eqref{eq pointwise calculation} over $j\in I$, $s\in S$, and $h\in H$, we see that \eqref{eq rewrite} equals
	\begin{align*}
	\pi_*\Big(\sum_{h,g,s,i,j}\big(h(\phi_i^g T\phi_j^{s}u)\big)(y_0)\Big)
	=\Big(\pi_*\sum_{g,s,i,j}\phi_i^{g}\cdot T(\phi_j^{s}u)\Big)(x),
	\end{align*}
	which is equal to $(\Psi(T)u)(x)$ by \eqref{eq folding} and noting Convention \ref{conv convention}.
\end{proof}
\begin{proposition}
\label{prop folding}
The folding map $\Psi$ defined in \eqref{eq folding} is a $*$-homomorphism
	\begin{align*}
	\Psi\colon\mathcal{B}_{\textnormal{fp}}(L^2(E_1))^{\Gamma_1}&\rightarrow \mathcal{B}_{\textnormal{fp}}(L^2(E_2))^{\Gamma_2}
	\end{align*}
and is independent of the choices of the set $S$ of coset representatives and compatible partitions of unity used to define it.
\end{proposition}
\begin{proof}
Boundedness of $\Psi(T)$ follows from boundedness and finite propagation of $T$, via the formula \eqref{eq pointwise folding}.
To see that $\Psi(T)$ is $\Gamma_2$-equivariant, note that by \eqref{eq pointwise folding}, we have for all $[\gamma]\in\Gamma_2$ and $u\in L^2(E_2)$ that
\begin{align*}
	([\gamma]\Psi(T)[\gamma^{-1}]u)(x)&=[\gamma]\cdot\pi_*\sum_{\substack{j\in I,\\g\in\Gamma_1}}T(\phi_j^g[\gamma^{-1}]u)(\gamma^{-1}y_0)\\
	&=[\gamma]\cdot\pi_*\sum_{\substack{j,g}}T\gamma^{-1}(\phi_j^{\gamma g}u)(\gamma^{-1}y_0).
\end{align*}
Using \eqref{eq identity4} and $\Gamma_1$-equivariance of $T$, this is equal to 
\begin{align*}
	\pi_*\sum_{j,g}\gamma T\gamma^{-1}(\phi_j^{\gamma g}u)(y_0)&=\pi_*\sum_{j,g}T(\phi_j^{\gamma g}u)(y_0)=(\Psi(T)u)(x),
\end{align*}
where we used a change-of-variable for the last equality. 

To see that $\Psi$ is a $*$-homomorphism, let $Q$ be another operator in $\mathcal{B}_{\textnormal{fp}}(L^2(E_1))^{\Gamma_1}$. Let $x\in M_2$ and $y_0\in\pi^{-1}(x)$. By \eqref{eq pointwise folding}, and the fact that both $T'$ and $T$ have finite propagation, one sees that there exists a finite subset $F\subseteq\Gamma_1$ such that
\begin{align*}
(\Psi(T'T)u)(x)&=\pi_*\sum_{\substack{j\in I,\\g\in F}}T'T(\phi_j^{g}u)(y_0)\nonumber\\
&=\pi_*\Big(T'\sum_{j,g}T(\phi_j^{g}u)\Big)(y_0)\nonumber\\
&=\pi_*\Big(T'\sum_{\substack{i\in I,\\\gamma\in F}}\phi_i^\gamma\sum_{j,g}T(\phi_j^{g}u)\Big)(y_0)\nonumber\\
&=\pi_*\Big(\sum_{i,\gamma}T'\big(\phi_i^\gamma\sum_{j,g}T(\phi_j^{g}u)\big)\Big)(y_0)\nonumber\\
&=(\Psi(T')\circ\Psi(T)u)(x).
\end{align*}
One checks directly that $\Psi$ respects $*$-operations.

That \eqref{eq folding} is independent of the choice of coset representatives $S\subseteq\Gamma_1$ is implied by the pointwise formula \eqref{eq pointwise folding}. To see that \eqref{eq folding} is independent of partitions of unity, let $\{\varphi_i\}$, $\{\varphi_i^{[\gamma]}\}$, and $\{\varphi_i^{\gamma}\}$ be another set of compatible partitions of unity for $N$, $M_2$, and $M_1$ with the same properties as $\{\phi_i\}$, $\{\phi_i^{[\gamma]}\}$, and $\{\phi_i^{\gamma}\}$. Let us write
$$\Psi^{\{\phi\}}(T)\quad\textnormal{and}\quad\Psi^{\{\varphi\}}(T)$$ 
to distinguish the operators defined by \eqref{eq folding} with respect to these two sets of partitions of unity. Since $T$ has finite propagation, there exists a finite subset $F'\subseteq\Gamma_1$ such that if $g\notin F$, then $\supp(\phi_j^g)\cap B_{\textnormal{prop}\,T}(y_0)=\emptyset$ and $\supp(\varphi_j^g)\cap B_{\textnormal{prop}\,T}(y_0)=\emptyset$ for any $j\in I$. By \eqref{eq pointwise folding}, for each $x\in M_2$ we have
\begin{align}
\label{eq POU}
(\Psi^{\{\phi\}}(T)u)(x)&=\pi_*\sum_{j\in I,\,g\in\Gamma_1}T(\phi_j^{g}u)(y_0)\nonumber\\
&=\pi_*\sum_{j\in I,\,g\in F}T(\mathbbm{1}_{B_{\textnormal{prop}\,T}(y_0)}\phi_j^{g}u)(y_0)\nonumber\\
&=\pi_*\Big(\big(T\sum_{\substack{j,g}}\mathbbm{1}_{B_{\textnormal{prop}\,T}(y_0)}\phi_j^{g}u\big)(y_0)\Big),
\end{align}
where $\mathbbm{1}_{B_{\textnormal{prop}\,T}(y_0)}$ is the characteristic function of the set $B_{\textnormal{prop}\,T}(y_0)$, and we have used that $T$ commutes with finite sums. Likewise, we have
\begin{align}
\label{eq POU 2}
(\Psi^{\{\varphi\}}(T)u)(x)&=\pi_*\Big(\big(T\sum_{\substack{j,g}}\mathbbm{1}_{B_{\textnormal{prop}\,T}(y_0)}\varphi_j^{g}u\big)(y_0)\Big).
\end{align}
Now observe that
$$\sum_{j, g}\mathbbm{1}_{B_{\textnormal{prop}\,T}(y_0)}\phi_j^{g}u\quad\textnormal{and}\quad\sum_{j, g}\mathbbm{1}_{B_{\textnormal{prop}\,T}(y_0)}\varphi_j^{g}u$$
are both equal to the lift of the section $u$ to $M_1$ restricted to the compact subset $B_{\textnormal{prop}\,T}(y_0)$. Thus the right-hand sides of \eqref{eq POU} and \eqref{eq POU 2} are equal, whence
\begin{align*}
(\Psi^{\{\phi\}}(T)(u))(x)
	=(\Psi^{\{\varphi\}}(T)(u))(x). &\qedhere
\end{align*}
\end{proof}
The next proposition shows that the folding map $\Psi$ preserves locality of operators. In particular, this means that $\Psi$ induces a map at the level of localization algebras (see Definition \ref{def folding localization}), which is a crucial ingredient for our main result, Theorem \ref{thm funct higher rho} .
\begin{proposition}
\label{prop prop}
For any operator $T\in\mathcal{B}_{\textnormal{fp}}(L^2(E_1))^{\Gamma_1}$, we have
$$\textnormal{prop}\big(\Psi(T)\big)\leq\textnormal{prop}(T).$$	
\end{proposition}
\begin{proof}
Take a section $u\in C_c(E_2)$, and write it as a finite sum $\sum_{j,[g]}\phi_j^{[g]} u$. By \eqref{eq pointwise folding},
$$\big(\Psi(T)u)(x)=\pi_*\sum_{j\in I, g\in\Gamma_1}T(\phi_j^{g}u)(y_0),$$ 
where $y_0$ is a lift of $x$ to $M_1$. Now if $d_{M_2}(x,\supp u)>\propagation(T)$, then
$$d_{M_1}\big(y_0,\supp(\phi_j^g u)\big)>\propagation(T)$$
for any $j\in I$ and $g\in\Gamma_1$, since distances between points do not increase under the projection $\pi$. Thus if $d_{M_2}(x,\supp u)>\propagation(T)$, then $\Psi(T)u(x)=0$. Since this holds for any section $u\in C_c(E_2)$, we have $\propagation(\Psi(T))\leq\propagation(T)$.
\end{proof}

\begin{proposition}
	\label{prop psi surj}
	The map $\Psi$ is surjective and preserves local compactness of operators. Moreover, it restricts to a surjective $*$-homomorphism
	$$\Psi\colon\mathbb{C}[M_1]^{\Gamma_1}\rightarrow\mathbb{C}[M_2]^{\Gamma_2}.$$
\end{proposition}
\begin{proof}
	Let $S\subseteq\Gamma_1$ be the set of coset representatives used in the definition of $\Psi$. For any $T_2\in\mathcal{B}_{\textnormal{fp}}(L^2(E_2))^{\Gamma_2}$, define an operator $T_1$ on $L^2(E_1)$ by the formula
	\begin{equation}
	\label{eq T1}
	T_1(v)\coloneqq\sum_{\substack{g\in\Gamma_1\\i\in I\,\,}}\sum_{\substack{s\in S\\j\in I}}\phi_j^{gs}T_2\big(\phi_i^{[g]} v\big),
	\end{equation}
	for $v\in L^2(E_1)$. 
	We claim that $T_1\in\mathcal{B}_{\textnormal{fp}}(L^2(E_1))^{\Gamma_1}$ and that $\Psi(T_1)=T_2$. Indeed, for each fixed $i$ and $j$, the sum of operators
	$$\sum_{\substack{g\in\Gamma_1}}\sum_{\substack{s\in S}}\phi_j^{gs}T_2\phi_i^{[g]}$$
	converges strongly in $\mathcal{B}(L^2(E_1))$. Since $I$ is a finite indexing set, it follows that $T_1$ is bounded. To see that $T_1$ is $\Gamma_1$-equivariant, note that for any $\gamma\in\Gamma_1$, we have, by the identities \eqref{eq identity2} and \eqref{eq identity4}, that	\begin{align*}
	\gamma^{-1}T_1\gamma(v)&=\sum_{g,i}\sum_{s,j}\gamma^{-1}\big(\phi_j^{gs}T_2(\phi_i^{[g]}\gamma v)\big)\\
	&=\sum_{g,i}\sum_{s,j}\gamma^{-1}\big(\phi_j^{gs}T_2(\pi_*(\gamma\cdot(\phi_i^{\gamma^{-1}g}v)))\big)\\
	&=\sum_{g,i}\sum_{s,j}\gamma^{-1}\big(\phi_j^{gs}T_2([\gamma]\cdot(\phi_i^{[\gamma^{-1}g]}v))\big),
	\end{align*}
	where we observe Convention \ref{conv convention} as usual. Using $\Gamma_2$-equivariance of $T_2$ and \eqref{eq identity1}, one finds this to be equal to
	\begin{align*}
	&\sum_{g,i}\sum_{s,j}\gamma^{-1}\big(\phi_j^{[gs]}([\gamma] T_2(\phi_i^{[\gamma^{-1}g]}v))\big)\\
	=&\sum_{g,i}\sum_{s,j}\phi_j^{\gamma^{-1}gs} T_2(\phi_i^{\gamma^{-1}g}v)\\
	=&\sum_{\gamma^{-1}g,i}\sum_{s,j}\phi_j^{\gamma^{-1}gs} T_2(\phi_i^{[\gamma^{-1}g]}v)\\
	=&\,\,T_1(v),
	\end{align*}
	using the definition of $T_1$ and a change of variable for the last equality. 
	
	Next, finite propagation of $T_1$ follows from finite propagation of $T_2$. Indeed, for any $v\in C_c(E_2)$, $g\in\Gamma_1$, and $i\in I$, the set
	$$S_{T_1}\coloneqq\left\{s\in S\,\big|\,\phi_j^{[gs]}T_2\big(\phi_i^{[g]} v\big)\neq 0\textnormal{ for some }j\in I\right\}$$
	is finite, thus the sum over $S$ in equation \eqref{eq T1} reduces to a finite sum over $S_{T_1}$. By the same equation, and the fact that the $\Gamma_1$-action is isometric, we have
	\begin{align*}
	\textnormal{prop}(T_1)&\leq\sup_{\substack{\,\,s\in S_{T_1}\\i,j\in I}}\{d_{M_1}\big(U_i^g,U_j^{gs}\big)\}=\sup_{\substack{\,\,s\in S_{T_1}\\i,j\in I}}\left\{d_{M_1}\big(U_i^e,U_j^{s}\big)\right\}<\infty.
	\end{align*}
	
	We now show that $\Psi(T_1)=T_2$. By \eqref{eq T1} and \eqref{eq folding complicated} we have, for any $u\in C_c(E_2)$,
	\begin{align*}
	\Psi(T_1)u&=\sum_{\substack{\gamma\in\Gamma_1,\\t\in S}}\sum_{i,j\in I}\phi_i^{[\gamma]}T_1\phi_j^{t}(u)\\
	&=\sum_{\substack{\gamma\in\Gamma_1,\\t\in S}}\sum_{i,j\in I}\phi_i^{[\gamma]}\Big(\sum_{\substack{g\in\Gamma_1\\k\in I\,\,}}\sum_{\substack{s\in S\\l\in I}}\phi_l^{gs}T_2\big(\phi_k^{[g]}(\phi_j^{t}u)\big)\Big).
	\end{align*}
	Finite propagation of $T_2$ and compact support of $u$ imply that all of these sums are finite, so by \eqref{eq identity3} this is equal to
	\begin{align*}
	\sum_{g,s,k,l}\sum_{\gamma,t,i,j}\phi_i^{[\gamma]}\big(\phi_l^{gs}T_2\big(\phi_k^{[g]}(\phi_j^{t}u)\big)\big)=\sum_{g,s,k,l}\sum_{t,j}\phi_l^{[gs]}T_2\big(\phi_k^{[g]}(\phi_j^{t}u)\big).
	\end{align*}
	Since $\{gs\,|\,s\in S\}$ is a set of coset representatives for $H\backslash\Gamma_1$ for any $g\in\Gamma_1$, the identity \eqref{eq identity3} implies that the above is equal to
	\begin{align*}
		\sum_{t,j}T_2\Big(\sum_{g,k}\phi_k^{[g]}(\phi_j^{t}u)\Big)=\sum_{t,j}T_2(\phi_j^{[t]}u)=T_2 u.
	\end{align*}
	Thus $\Psi(T_1)=T_2$ as bounded operators on $L^2(E_2)$.
	
	Finally, that $\Psi$ preserves local compactness of operators follows directly from \eqref{eq pointwise folding}, so that $\Psi$ restricts to a map $\mathbb{C}[M_1]^{\Gamma_1}\rightarrow\mathbb{C}[M_2]^{\Gamma_2}$. To see that this map is surjective, suppose $T_2$ is locally compact. Then for any $f\in C_c(M_1)$, we have
	\begin{equation*}
	f T_1=f\sum_{\substack{g\in\Gamma_1\\s\in S\,\,}}\sum_{\substack{i,j\in I}}\phi_j^{gs}T_2\phi_i^{[g]}.
	\end{equation*}
	The sum over $\Gamma_1$ reduces to a sum over the finite set
	$$F_f\coloneqq\left\{g\in\Gamma_1\,\big|\,U_j^{gs}\cap\supp(f)\neq\emptyset\textnormal{ for some }j\in I\right\},$$
	hence $f T_1$ is equal to
	\begin{align*}
	\sum_{\substack{g\in F_f\\s\in S\,\,}}\sum_{\substack{i,j\in I}}f\phi_j^{gs}\mathbbm{1}_{\pi_*(\supp f)}T_2\phi_i^{[g]}.
	\end{align*}
	Since the subset $\pi_*(\supp f)\subseteq M_2$ is compact, local compactness of $T_2$ means that this is a finite sum of compact operators and hence compact. A similar argument shows that the operator $T_1 f$ is compact. Thus the operator $T_1$ is locally compact, so $\Psi$ restricts to a surjective $*$-homomorphism $\mathbb{C}[M_1]^{\Gamma_1}\rightarrow\mathbb{C}[M_2]^{\Gamma_2}$.
\end{proof}

\hfill\vskip 0.1in
\subsection{Description of the folding map on invariant sections}
\label{subsec desc op}
 	\hfill\vskip 0.05in
	\noindent The folding map $\Psi$ admits an equivalent description using $N$-invariant sections of $E_1$. Such sections, while not in general square-integrable over $M_1$, form a space that is naturally isomorphic to $L^2(E_2)^{\Gamma_2}$, as we now describe. This allows computations involving $\Psi$ to be carried out entirely on $M_1$. Thus it may be a useful perspective for some applications. The discussion below slightly generalizes the averaging map from \cite[subsection 5.2]{GHM3} applied in the context of discrete groups.

Let $c\colon M_1\rightarrow [0,1]$ be a function whose support has compact intersections with every $H$-orbit, and such that for all $x \in M_1$,
\begin{equation*}
\sum_{h\in H}c(h x)^2 = 1.
\end{equation*}
Note that this sum is finite by properness of the $\Gamma_1$-action. 

Let $C_{\tc}(E_1)^{H}$ denote the space of \emph{$H$-transversally compactly supported} sections of $E_1$, defined as the space of continuous, $H$-invariant sections of $E_1$ whose supports have compact images in $M_2\cong M_1/H$ under the quotient map. Let $L^2_T(E_1)^{H}$ denote the Hilbert space of $H$-invariant, \emph{$H$-transversally $L^2$-sections of $E_1$}, defined as the completion of $C_{\tc}(E_1)^{H}$ with respect to the inner product
\[
(s_1, s_2)_{L^2_T(E_1)^N} \coloneqq (c s_1, c s_2)_{L^2(E_1)}.
\]
Then one checks that:
\begin{lemma}
The space $L^2_T(E_1)^{H}$ is naturally unitarily isomorphic to $L^2(E_2)$ and is independent of the choice of $c$.
\end{lemma}

For any $T\in\mathcal{B}_{\textnormal{fp}}(L^2(E_1))^{\Gamma_1}$, the operator $\Psi(T)$ on $L^2(E_2)$ from \eqref{eq folding} can be described equivalently by its action on the isomorphic space $L^2_T(E_1)^H$ as follows. Take a partition of unity $\{\phi_i^g\}$ of $M_1$ as in \eqref{eq POU M1}. Let $s\in L^2_T(E_1)^{H}$ be an $H$-transversally $L^2$-section of $E_1$. For any $y\in M_1$, set
\begin{equation} 
\label{eq fold T}
(\Psi(T)s)(y) \coloneqq \sum_{i\in I, g\in\Gamma_1}\left(T(\phi_i^g s\right)(y)).
\end{equation}
The fact that $T$ has finite propagation means that this pointwise sum is finite, and that $c(\Psi(T)s)\in L^2(E_1)$. Further, $\Psi(T)s$ is an element of $L^2_T(E_1)^H$, since for any $h\in H$ and $y\in M_1$ we have
\[
(h(\Psi(T)s))(y) = \sum_{i\in I, g\in\Gamma_1}\left(T\circ h(\phi_i^g s\right)(y))=\sum_{i\in I, g\in\Gamma_1}T(\phi_i^{hg} s)(y)=(\Psi(T)s)(y),
\]
where we have used the equivariance properties of $T$ and $s$. Finally, a direct comparison shows that this definition is equivalent with that given by equation \eqref{eq pointwise folding}.

\hfill\vskip 0.1in	
\subsection{Induced maps on geometric $C^*$-algebras and $K$-theory}
\label{subsec induced}
 	\hfill\vskip 0.05in
	\noindent By Proposition \ref{prop psi surj}, the folding map $\Psi$ restricts to a surjective $*$-homomorphism 
		$$\Psi\colon\mathbb{C}[M_1]^{\Gamma_1}\rightarrow\mathbb{C}[M_2]^{\Gamma_2}.$$
	By the defining property of maximal completions of $*$-algebras, $\Psi$ extends to a surjective $*$-homomorphism between maximal equivariant Roe algebras,
$$\Psi\colon C^*_{\textnormal{max}}(M_1)^{\Gamma_1}\rightarrow C^*_{\textnormal{max}}(M_2)^{\Gamma_2}.$$

	In this paper we will make use of a number natural extensions of this map to other geometric $C^*$-algebras, as well as the induced maps on $K$-theory. We refer to these maps collectively as \emph{folding maps}.

To begin, we have the following elementary lemma:
\begin{lemma}
	\label{lem surjective}
	Let $A$ and $B$ be $C^*$-algebras and let $\phi\colon A\rightarrow B$ be a surjective $*$-homomorphism. Then $\phi$ extends to a $*$-homomorphism $\widetilde\phi\colon \mathcal{M}(A)\rightarrow\mathcal{M}(B)$.
\end{lemma}
\begin{proof}
	Since $\phi$ is surjective, we may define $\widetilde{\phi}$ by requiring 
	$$\widetilde{\phi}(m)\phi(a)=\phi(ma),\quad\phi(a)\widetilde{\phi}(m)=\phi(am),$$ 
	for $m\in\mathcal{A}$. Concretely, by way of a faithful non-degenerate representation $\sigma\colon B\rightarrow\mathcal{B}(H)$ on some Hilbert space $H$, we may identify $B$ with a subalgebra of $\mathcal{B}(H)$. Since $\phi$ is surjective, the composition $\sigma\circ\phi\colon A\rightarrow\mathcal{B}(H)$ is a non-degenerate representation of $A$. Thus the above formulas define a representation of $\mathcal{M}(A)$ with values in the idealizer of $\sigma(B)$, which one identifies with $\mathcal{M}(B)$. (See also \cite{Davidson} Lemma I.9.14.)
\end{proof}

By Lemma \ref{lem surjective} and Proposition \ref{prop psi surj}, the map $\Psi$ extends to a $*$-homomorphism 
$$\mathcal{M}\left(C^*_{\textnormal{max}}(M_1)^{\Gamma_1}\right)\rightarrow\mathcal{M}\left(C^*_{\textnormal{max}}(M_2)^{\Gamma_2}\right).$$
Viewing $C^*_{\textnormal{max}}(M_j)^{\Gamma_j}$ as an ideal in $\mathcal{M}$, we will still denote this extended map by $\Psi$. This map will be essential when we apply the functional calculus for the maximal Roe algebra in sections \ref{sec higher ind} and \ref{sec higher rho}.

Next, suppose we have a path $r\colon [0,\infty)\to C^*_{\textnormal{max}}(M_1)^{\Gamma_1}$ that satisfies
$$\textnormal{prop}(r(t))\to 0\quad\textnormal{as}\quad t\to\infty.$$ 
Then by Proposition \ref{prop prop}, the same is true of the path $\Psi\circ r\colon [0,\infty)\rightarrow C^*_{\textnormal{max}}(M_1)^{\Gamma_2}$. This allows the folding map to pass to the level of localization algebras:
\begin{definition}
\label{def folding localization}
Define the map
\begin{align*}
\Psi_L\colon C^*_{L,\textnormal{max}}(M_1)^{\Gamma_1}&\to C^*_{L,\textnormal{max}}(M_2)^{\Gamma_2}\\
r&\mapsto\Psi\circ r.
\end{align*}
Then $\Psi_L$ restricts to a $*$-homomorphism between obstruction algebras:
$$\Psi_{L,0}\colon C^*_{L,0,\textnormal{max}}(M_1)^{\Gamma_1}\to C^*_{L,0,\textnormal{max}}(M_2)^{\Gamma_2}.$$
\end{definition}

Each of the above maps $\Psi$, $\Psi_L$, and $\Psi_{L,0}$ induces a map at the level of $K$-theory. In this paper, we will make use of two of them:
\begin{align*}
	\Psi_*\colon K_{\bullet}\left(C^*_{\textnormal{max}}(M_1)^{\Gamma_1}\right)&\rightarrow K_{\bullet}\left(C^*_{\textnormal{max}}(M_2)^{\Gamma_2}\right),\\
	(\Psi_{L,0})_*\colon K_{\bullet}\left(C^*_{L,0,\textnormal{max}}(M_1)^{\Gamma_1}\right)&\rightarrow K_{\bullet}\left(C^*_{L,0,\textnormal{max}}(M_2)^{\Gamma_2}\right).
\end{align*}
\begin{remark}
In the recent paper \cite{SchickStudent}, Schick and Seyedhosseini have constructed a slightly different completion of the equivariant algebraic Roe algebra, called the \emph{quotient completion}, which still behaves functorially under maps between covering spaces. For details of the construction, see \cite[section 3]{SchickStudent}.
\end{remark}

\hfill\vskip 0.2in
\section{Functional calculus and the wave operator}
\label{sec funct calc}
We now develop the analytical properties of the wave operator on the maximal Roe algebra that will form the basis of our work in subsequent sections.

Let us begin by recalling the functional calculus for the maximal Roe algebra, which was developed in a more general setting in \cite{GXY}. The discussion here is specialized to the cocompact setting.

Throughout this section we will work in the geometric situation in subsection \ref{subsec geom setup}. To simplify notation, in this subsection and the next we will write $M,\Gamma$ for either $M_1,\Gamma_1$ or $M_2,\Gamma_2$, and $D,E$ for either $D_1,E_1$ or $D_2,E_2$. In other words, $\Gamma$ acts freely and properly on $M$ with compact quotient, and $D$ is a $\Gamma$-equivariant operator on the bundle $E\to M$.

\hfill\vskip 0.1in
\subsection{Functional calculus on the maximal Roe algebra}
\label{subsec funct Roe}
\hfill\vskip 0.05in
\noindent We shall view the $C^*$-algebra $C^*_{\text{max}}(M)^{\Gamma}$ as a right Hilbert module over itself. The inner product and right action on $C^*_{\text{max}}(M)^{\Gamma}$ are defined naturally through multiplication: for $a,b\in C^*_{\text{max}}(M)^{\Gamma}$,
\begin{equation}
\label{eq Hilbert module structure}	
\langle a,b\rangle=a^*b,\qquad a\cdot b=ab.
\end{equation}
The algebra of compact operators on this Hilbert module can be identified with $C^*_{\text{max}}(M)^{\Gamma}$ via left multiplication. Similarly, the algebra of bounded adjointable operators can be identified with the multiplier algebra $\mathcal{M}$ of $C^*_{\text{max}}(M)^{\Gamma}$.

The operator $D$ defines an unbounded operator on this Hilbert module in the following way.
First note that $D$ acts on smooth sections of the external tensor product $E\boxtimes E^*\rightarrow M\times M$ by taking a section $s$ to the section $Ds$ defined by
\begin{equation}
\label{eq D action on kernels}
(Ds)(x,y)= D_{x} s(x,y),
\end{equation}
where $D_{x}$ denotes $D$ acting on the $x$-variable. 

Let $\mathcal{S}^\Gamma$ be the $*$-subalgebra of $\mathbb{C}[M]^\Gamma$ defined in Remark \ref{rem kernel algebra}. That is, an element of $\mathcal{S}^\Gamma$ is an operator $T_{\kappa}$ given by a smooth kernel $\kappa\in C_b^\infty(E\boxtimes E^*)$ that:
\begin{enumerate}[(i)]
\item is $\Gamma$-equivariant with respect to the diagonal $\Gamma$-action;
\item has finite propagation, meaning that there exists a constant $c_{\kappa}\geq 0$ such that if $d(x,y)>c_{\kappa}$ then $\kappa(x,y)=0$. 
\end{enumerate}

The operator $D$ acts on elements of $\mathcal{S}^\Gamma$ by acting on the corresponding smooth kernels as in \eqref{eq D action on kernels}. In this way, $D$ becomes a densely defined operator on the Hilbert module $C^*_{\text{max}}(M)^{\Gamma}$ that one verifies is symmetric with respect to the inner product in \eqref{eq Hilbert module structure}. 


\begin{theorem}[\cite{GXY} Theorem 3.1]
	\label{thm:regularity}
	There exists a real number $\mu\neq 0$ such that the operators
	$$D\pm\mu i\colon C^*_{\text{max}}(M)^{\Gamma}\rightarrow C^*_{\text{max}}(M)^{\Gamma}$$
	have dense range.
\end{theorem}
Consequently, the operator $D$ on the Hilbert module $C^*_{\text{max}}(M)^{\Gamma}$ is regular and essentially self-adjoint, and so admits a continuous functional calculus (see \cite[Theorem 10.9]{Lance} and \cite[Proposition 16]{Kucerovsky}):
\begin{theorem}\label{thm:functionalcalculus}
	For $j=1$ or $2$, there is a $*$-preserving linear map
	\begin{align*}
	\pi\colon C(\mathbb{R})&\rightarrow\mathcal{R}\left(C^*_{\text{max}}(M)^{\Gamma}\right),\\
	f&\mapsto f(D)\coloneqq\pi(f),
	\end{align*}
	where $C(\mathbb{R})$ denotes the continuous functions $\mathbb{R}\to\mathbb{C}$, and $\mathcal{R}\left(C^*_{\text{max}}(M)^{\Gamma}\right)$ denotes the regular operators on $C^*_{\text{max}}(M)^{\Gamma}$, such that:	
	\begin{enumerate}
		\item[(i)] $\pi$ restricts to a $*$-homomorphism $C_b(\mathbb{R})\rightarrow\mathcal{M}$;		
		\item[(ii)] If $|f(t)|\leq|g(t)|$ for all $t\in\mathbb{R}$, then $\textnormal{dom}(g(D))\subseteq\textnormal{dom}(f(D))$;
		\item[(iii)] If $(f_n)_{n\in\mathbb{N}}$ is a sequence in $C(\mathbb{R})$ for which there exists $f'\in C(\mathbb{R})$ such that $|f_n(t)|\leq |f'(t)|$ for all $t\in\mathbb{R}$, and if $f_n$ converge to a limit function $f \in C(\R)$ uniformly on compact subsets of $\mathbb{R}$, then $f_n(D)x\to f(D)x$ for each $x\in\textnormal{dom}(f(D))$;
		\item[(iv)] $\textnormal{Id}(D)=D$.
	\end{enumerate}
\end{theorem}
	
\hfill\vskip 0.1in
\subsection{The wave operator}
\label{subsec wave}
\hfill\vskip 0.05in
\noindent We now discuss the relationship between the wave operator formed using the functional calculus from Theorem \ref{thm:functionalcalculus} and the classical wave operator on $L^2$. Both of these operators can be viewed as bounded multipliers of the maximal Roe algebra $C^*_{\textnormal{max}}(M)^{\Gamma}$, and we will see that:
\begin{proposition}
\label{prop waves equal}
For each $t\in\mathbb{R}$, we have
$$e^{itD}_{L^2}=e^{itD}\in\mathcal{M}(C^*_{\textnormal{max}}(M)^{\Gamma}).$$
\end{proposition}

Let us begin by explaining both sides of the equation, starting with the right-hand side. For each $t\in\mathbb{R}$, the functional calculus from Theorem \ref{thm:functionalcalculus} allows one to form a bounded adjointable operator $e^{itD}$ on the Hilbert module $C^*_{\textnormal{max}}(M)^{\Gamma}$. The resulting group of operators $\{e^{itD}\}_{t\in\mathbb{R}}$ is strongly continuous in the sense of Theorem \ref{thm:functionalcalculus} (iii) and uniquely solves the wave equation on $C^*_{\text{max}}(M)^{\Gamma}$:
\begin{lemma}
		\label{lem max wave}
		For any $\kappa\in\mathcal{S}^{\Gamma}$, $u(t)=e^{itD}\kappa$ is the unique solution of the problem
		\begin{equation}
		\label{eq max wave}
		\frac{du}{dt}=iD u,\quad u(0)=\kappa,
		\end{equation}
		with $u\colon\mathbb{R}\to C^*_{\text{max}}(M)^{\Gamma}$ a differentiable map taking values in $\textnormal{dom}(D)$.
\end{lemma}
	\begin{proof}
		For each $t\in\mathbb{R}$, the function $s\mapsto e^{its}$ is a unitary in $C_b(\mathbb{R})$. Hence $e^{itD}$ is bounded adjointable and unitary. Let $h_n$ be a sequence of positive real numbers converging to $0$ as $n\rightarrow\infty$. Then for each $t\in\mathbb{R}$, the sequence of functions
		$$f_n(s)\coloneqq\frac{e^{i(t+h_n)s}-e^{its}}{h_n}$$ converges to $f(s)\coloneqq ise^{its}$ uniformly on compact subsets of $\mathbb{R}$ in the limit $n\to\infty$. Also, each $f_n$ is bounded above by $|1+s|$. By Theorem \ref{thm:functionalcalculus} (iii), this implies \eqref{eq max wave}.
		
		For the uniqueness claim, let $v$ be another solution of \eqref{eq max wave} with $v(0)=\kappa$. For any fixed $s\in\mathbb{R}$ and $0\leq t\leq s$, set $w(t)=e^{itD}v(s-t)$. Then we have
		$$\frac{dw}{dt}=iDe^{itD}v(s-t)-ie^{itD}Dv(s-t)=0.$$
		It follows that $w(t)$ is constant for all $t$, hence 
		\begin{align*}
		v(s)=w(0)=w(s)=e^{isD}\kappa=u(s).&\qedhere
		\end{align*}
	\end{proof}

On the other hand, we may apply the functional calculus on $L^2(E)$ to the essentially self-adjoint operator 
\begin{equation}
\label{eq classical D}
D\colon L^2(E)\to L^2(E)
\end{equation}
to form the \emph{classical} wave operator $e^{itD}_{L^2}$, not to be confused with with the bounded adjointable operator $e^{itD}$. The resulting operator group $\{e^{itD}_{L^2}\}_{t\in\mathbb{R}}$ is strongly continuous in $\mathcal{B}(L^2(E))$ and uniquely solves the wave equation on $L^2(E)$: for every $v\in C_c^\infty(E)$, we have
\begin{equation*}
\label{eq wave L2}
	\frac{d}{dt}(e^{itD}_{L^2}(v))=iD e^{itD}_{L^2}(v).
\end{equation*}
Via composition, $e^{itD}_{L^2}$ defines a map
\begin{equation}
\label{eq wave on SGamma}
e^{itD}_{L^2}\colon \mathcal{S}^{\Gamma}\to\mathcal{S}^{\Gamma}.
\end{equation}
A standard argument involving the Sobolev embedding theorem now shows that for any $T_{\kappa}\in\mathcal{S}^\Gamma$, the kernel of $e^{itD}_{L^2}\circ T_\kappa$ is smooth, in addition to having finite propagation and being $\Gamma$-equivariant. Indeed, for any $m\in\mathbb{N}$, the operator $D^m\circ T_{\kappa}$ is bounded on $L^2(E)$, as $M/\Gamma$ is compact. Together with the fact that
$$D^m\circ e^{itD}_{L^2}\circ T_\kappa=e^{itD}_{L^2}\circ D^m\circ T_\kappa,$$
Sobolev theory implies that the image of $e^{itD}_{L^2}\circ T_\kappa$ lies in the smooth sections, so that this operator also has smooth kernel. By Remark \ref{rem max norm}, \eqref{eq wave on SGamma} extends uniquely to a bounded multiplier of $C^*_{\textnormal{max}}(M)^{\Gamma}$,
\begin{equation}
\label{eq L2 multiplier}	
e^{itD}_{L^2}\colon C^*_{\textnormal{max}}(M)^{\Gamma}\to C^*_{\textnormal{max}}(M)^{\Gamma}.
\end{equation}
This is the operator on the left-hand side of Proposition \ref{prop waves equal}. 

Observe that 
if $k$ is a smooth, \emph{compactly supported} Schwartz kernel on $M\times M$, then we have the pointwise equality
\begin{equation}
\label{eq pointwise cpt}
\Big(\frac{d}{dt}e^{itD}_{L^2}k\Big)(x,y)=(iDe^{itD}_{L^2}k)(x,y),
\end{equation}
where $e^{itD}_{L^2}k$ denotes the smooth Schwartz kernel of the composition $e^{itD}_{L^2}\circ T_k$. The proof of Proposition \ref{prop waves equal}, involves a more general form of this observation for kernels in $\mathcal{S}^\Gamma$:


\begin{lemma}
\label{lem pointwise wave}
Let $\kappa\in\mathcal{S}^\Gamma$, and let $e^{itD}_{L^2}\kappa$ denote the smooth Schwartz kernel of $e^{itD}_{L^2}\circ T_\kappa$. Then for every $t\in\mathbb{R}$ and $x,y\in M$, we have the pointwise equality
\begin{equation}
\label{eq pointwise cocpt}
\Big(\frac{d}{dt}e^{itD}_{L^2}\kappa\Big)(x,y)=(iDe^{itD}_{L^2}\kappa)(x,y).
\end{equation}
Moreover, the path $t\mapsto e^{itD}_{L^2}\kappa$ is continuous with respect to the operator norm, and
\begin{equation}
\label{eq wave kernel}
\lim_{h\to 0}\,\,\,\bigg\|\frac{e^{i(t+h)D}_{L^2}\circ T_\kappa-e^{itD}_{L^2}\circ T_\kappa}{h}-iD e^{itD}_{L^2}\circ T_\kappa\bigg\|_{\mathcal{B}(L^2(E))}= 0.
\end{equation}
\end{lemma}
\begin{proof}
	Fix $t\in\mathbb{R}$, $\epsilon>0$, and $x$, $y\in M$. Take $\phi\in C_c^\infty(M)$ such that $\phi(z)=1$ if $z\in B_{2t}(x)$, and $\phi(z)=0$ if $z\in M\backslash B_{3t}(x)$. Since $\kappa$ has finite propagation, the smooth kernel $\phi\kappa$ is compactly supported in $M\times M$, so by \eqref{eq pointwise cpt} we have
\begin{equation}
\label{eq phi kappa}
	\Big(\frac{d}{dt}e^{itD}_{L^2}\phi\kappa\Big)(x,y)=(iDe^{itD}_{L^2}\phi\kappa)(x,y).
\end{equation}
	The fact that $e^{itD}_{L^2}$ has propagation at most $t$, together with the fact that 
	$$(1-\phi)\kappa(z,y)=0$$ 
	for all $z\in B_{2t}(x)$, implies that for all $w\in B_t(x)$ we have $(e^{itD}_{L^2}(1-\phi)\kappa)(w,y)=0$. It follows that 
	\begin{equation}
	(e^{itD}_{L^2}\kappa)(w,y) = (e^{itD}_{L^2}\phi\kappa)(w,y).
\end{equation}
Taking $w=x$ and combining with \eqref{eq phi kappa} gives \eqref{eq pointwise cocpt}.

To establish norm continuity, it suffices to show that $t\mapsto e^{itD}_{L^2}\kappa$ is norm-continuous at $t=0$. For each $x$, $y$, and $t$ in an interval $[-t_0,t_0]$, we have
\begin{equation}
\label{eq mean value}
\Big|e^{itD}_{L^2}\kappa(x,y)-\kappa(x,y)\Big|\leq |t|\cdot\sup_{s\in[-t_0,t_0]}\Big|iDe^{isD}_{L^2}\kappa(x,y)\Big|
\end{equation}
by the mean value theorem applied to \eqref{eq pointwise cocpt}. Since the operator $iDe^{isD}_{L^2}\kappa$ is $\Gamma$-equivariant with finite propagation, and $M$ is cocompact, there exists a compact subset $K\subseteq M$ such that
\begin{equation}
\label{eq C t0}
\sup_{\substack{x,y\in M,\\s\in[-t_0,t_0]}}\Big|iDe^{isD}_{L^2}\kappa(x,y)\Big|=\sup_{\substack{x,y\in K,\\s\in[-t_0,t_0]}}\Big|iDe^{isD}_{L^2}\kappa(x,y)\Big|\leq C_{t_0}
\end{equation}
for some constant $C_{t_0}$. Now since $e^{itD}_{L^2}\kappa-\kappa$ has finite propagation, its operator norm can be estimated using and \eqref{eq mean value} and \eqref{eq C t0}:
$$\norm{e^{itD}_{L^2}\kappa-\kappa}_{\mathcal{B}(L^2(E))}\leq C\cdot\sup_{x,y\in M}\Big|e^{itD}_{L^2}\kappa(x,y)-\kappa(x,y)\Big|\leq CC_{t_0}|t|,$$
for some constant $C$ depending on the propagation of $\kappa$. We obtain norm continuity by taking a limit $t\to 0$. The proof of \eqref{eq wave kernel} is a straightforward adaptation of this argument.
\end{proof}
With these preparations, we now prove Proposition \ref{prop waves equal}.

%


\begin{proof}[Proof of Proposition \ref{prop waves equal}]
Fix $t\in\mathbb{R}$. Let $T_\kappa\in\mathcal{S}^{\Gamma}$ with smooth Schwartz kernel $\kappa$. 
	We claim that the kernel $e^{itD}_{L^2}\kappa$ satisfies the wave equation in $C^*_{\textnormal{max}}(M)^{\Gamma}$. To see this, first note that by equation \eqref{eq wave kernel} we have  
$$\lim_{h\to 0}\,\,\,\bigg\|\frac{e^{i(t+h)D}_{L^2}\circ T_\kappa-e^{itD}_{L^2}\circ T_\kappa}{h}-iD e^{itD}_{L^2}\circ T_\kappa\bigg\|_{\mathcal{B}(L^2(E))}= 0.$$
Now since the kernels $e^{i(t+h)D}_{L^2}\kappa$ and $D e^{itD}_{L^2}\kappa$ each have propagation at most 
$$r\coloneqq\textnormal{prop}(\kappa)+|t+h|,$$
by Remark \ref{rem max norm} there exists a constant $C_r$ such that the norm of
$$\frac{e^{i(t+h)D}_{L^2}\kappa-e^{itD}_{L^2}\kappa}{h}-iD e^{itD}_{L^2}(\kappa)$$
in $C^*_{\textnormal{max}}(M)^{\Gamma}$ is bounded above by $C_r$ times its norm in $\mathcal{B}(L^2(E))$. Thus
$$\lim_{h\to 0}\,\,\,\bigg\|\frac{e^{i(t+h)D}_{L^2}\kappa-e^{itD}_{L^2}\kappa}{h}-iD e^{itD}_{L^2}(\kappa)\bigg\|_{\textnormal{max}}= 0,$$
so the operator group $\{e^{itD}_{L^2}\}_{t\in\mathbb{R}}$ solves the wave equation \eqref{eq max wave} in $C^*_{\textnormal{max}}(M)^{\Gamma}$. It follows from the uniqueness property in Lemma \ref{lem max wave} that $e^{itD}_{L^2}$ and $e^{itD}$ coincide on $\mathcal{S}^{\Gamma}$, and hence are equal as elements of $\mathcal{M}$.
\end{proof}

\hfill\vskip 0.2in
\section{Functoriality for the higher index}
\label{sec higher ind}
\noindent
In this section, we discuss functoriality in the case of the maximal higher index from the point of view of folding maps and wave operators. This will prepare us for the proof of our main result, Theorem \ref{thm funct higher rho}, in the next section.

%

Suppose $M_1$ and $M_2$ are two Galois covers of $N$ with deck transformation groups $\Gamma_1$ and $\Gamma_2\cong \Gamma_1/H$, for some normal subgroup $H$ of $\Gamma_1$. Then the quotient homomorphism $\Gamma_1\rightarrow\Gamma_2$ induces a natural surjective $*$-homomorphism between group algebras,
$$\alpha\colon \mathbb{C}\Gamma_1\rightarrow\mathbb{C}\Gamma_2,\quad\sum_{i=1}^k c_{\gamma_i}\gamma_i\mapsto\sum_{i=1}^k c_{\gamma_i}[\gamma_i],$$
where $[\gamma]$ is the class of $\gamma\in\Gamma_1$ in $\Gamma_1/H$.
In this setting, we have the following functoriality result for the maximal higher index, which follows from a theorem of Valette \cite[Theorem 1.1]{MislinValette}:
\begin{theorem}
	\label{thm funct higher ind}
	Let $N$ be a closed Riemannian manifold. Let $D_N$ be a first-order, self-adjoint elliptic differential operator acting on a bundle $E_N\rightarrow N$. Let $M_1$ and $M_2$ be Galois covers of $N$ with deck transformation groups $\Gamma_1$ and $\Gamma_2\cong\Gamma_1/H$ respectively, for a normal subgroup $H$ of $\Gamma_1$. Let $D_1$ and $D_2$ be the lifts of $D_N$ to $M_1$ and $M_2$. Then the map on $K$-theory induced by $\alpha$ relates the maximal higher indices of $D_1$ and $D_2$:
	$$\alpha_*(\Ind_{\Gamma_1,\textnormal{max}}D_1)=\Ind_{\Gamma_2,\textnormal{max}}D_2\in K_\bullet\big(C^*_{\textnormal{max}}(\Gamma_2)\big).$$
\end{theorem}
This result was proved originally using $KK$-theory. To prepare for the proof of Theorem \ref{thm funct higher rho}, we now give a proof of Theorem \ref{thm funct higher ind} using local properties of the wave operator and the folding map $\Psi$ from the previous sections.


\subsection{Higher index}
\label{subsec higher ind}
\hfill\vskip 0.05in
	\noindent We begin by recalling the definition of the maximal higher index of an equivariant elliptic operator. 
Let $\Gamma$, $M$, and $D$ be as in section \ref{sec funct calc}. 

Let $\mathcal{Q}\coloneqq\mathcal{M}/C^*_\textnormal{max}(M)^{\Gamma}.$ Consider the short exact sequence of $C^*$-algebras
	$$0\rightarrow C^*_\textnormal{max}(M)^{\Gamma}\rightarrow\mathcal{M}\rightarrow\mathcal{Q}\rightarrow 0,$$
	where $\mathcal{M}$ is shorthand for the multiplier algebra $\mathcal{M}\big(C^*_\textnormal{max}(M)^{\Gamma}\big)$. This induces the following six-term exact sequence in $K$-theory:
	\[
	\begin{tikzcd}
	K_0(C^*_\textnormal{max}(M)^{\Gamma}) \ar{r} & K_0(\mathcal{M}) \ar{r} & K_0(\mathcal{Q}) \ar{d}{\partial_1} \\
	K_1(\mathcal{Q}) \ar{u}{\partial_0} & K_1(\mathcal{M}) \ar{l} & K_1(C^*_\textnormal{max}(M)^{\Gamma}) \ar{l},
	\end{tikzcd}
	\]
	where the connecting maps $\partial_0$ and $\partial_1$, known as \textit{index maps}, are defined as follows.
	\begin{definition}
		\label{def connectingmaps}
		\hfill
		\begin{enumerate}[(i)]
			\item $\partial_0$: let $u$ be an invertible matrix over $\mathcal{Q}$ representing a class in $K_1(\mathcal{Q})$. Let $v$ be the inverse of $u$. Let $U$ and $V$ be lifts of $u$ and $v$ to a matrix algebra over $\mathcal{M}$.
			Then the matrix
			\[W=
			\begin{pmatrix} 
			1 & 0\\ 
			U & 1
			\end{pmatrix}
			\begin{pmatrix} 
			1 & -V\\ 
			0 & 1
			\end{pmatrix}
			\begin{pmatrix} 
			1 & 0\\ 
			U & 1
			\end{pmatrix}
			\]
			is invertible, and
			$
			P=W\begin{pmatrix}
			1 & 0\\ 
			0 & 0
			\end{pmatrix}W^{-1}
			$	    
			is an idempotent. 
			We define
			\begin{equation}
			\label{eq even index}
			\partial_0[u]\coloneqq 
			\left[P
			\right]-
			\begin{bmatrix}
			0 & 0\\
			0 & 1
			\end{bmatrix}
			\in K_0\big(C^*_\textnormal{max}(M)^{\Gamma}\big).
			\end{equation}

			\item $\partial_1$: let $q$ be an idempotent matrix over $\mathcal{Q}$ representing a class in $K_0(\mathcal{Q})$. Let $Q$ be a lift of $q$ to a matrix algebra over $\mathcal{M}$. Then $e^{2\pi iQ}$ is a unitary in the unitized algebra, and we define
			\begin{equation}
			\label{eq odd index}
			\partial_1[q]\coloneqq\left[e^{2\pi iQ}\right]\in K_1(C^*_\textnormal{max}(M)^{\Gamma}).
			\end{equation}
		\end{enumerate}
	\end{definition}
	
	This construction is applied to the operator $D$ via the functional calculus from Theorem \ref{thm:functionalcalculus}, as follows. Let $\chi\colon\mathbb{R}\rightarrow\mathbb{R}$ be a continuous, odd function such that 
	$$\lim_{x\rightarrow +\infty}\chi(x)=1,$$ 
	known as a \emph{normalizing function}. Using Theorem \ref{thm:functionalcalculus}, we obtain an element $\chi(D)$ in $\mathcal{M}$. We now have:
	\begin{lemma}
		\label{lem chiD invertible}
	The class of $\chi(D)$ in $\mathcal{M}/C^*_\textnormal{max}(M)^{\Gamma}$ is invertible and independent of the choice of normalizing function $\chi$.
	\end{lemma}
	\begin{proof}
				Let $\mathcal{S}(\mathbb{R})$ denote the Schwartz space of functions $\mathbb{R}\rightarrow\mathbb{C}$. Then for every $f\in\mathcal{S}(\mathbb{R})$ with compactly supported Fourier transform $\widehat{f}$, the operator $f(D)$ is given by a smooth kernel \cite[Proposition 2.10]{Roe}. Since every $f\in C_0(\mathbb{R})$ function is a uniform limit of such functions, the first part of Theorem \ref{thm:functionalcalculus} implies that for every such $f$ we have $f(D)\in C^*_\textnormal{max}(M)^{\Gamma}$.
				
				Now if $\chi$ is a normalizing function, then $\chi^{2}-1\in C_0(\mathbb{R})$. Hence the class of $\chi(D)$ in $\mathcal{M}/C^*_\textnormal{max}(M)^{\Gamma}$ is invertible. Since any two normalizing functions differ by an element of $C_0(\mathbb{R})$, this class is independent of the choice of $\chi$.
	\end{proof}
	Using this lemma, one computes that $$\frac{\chi(D)+1}{2}$$ is an idempotent modulo $C^*_\textnormal{max}(M)^{\Gamma}$ and so defines element of $K_0\big(\mathcal{M}/C^*_\textnormal{max}(M)^{\Gamma}\big)$. This leads us to the definition of the maximal higher index of $D$:
		\begin{definition}
			For $i=1,2$, let $\partial_i$ be the connecting maps from Definition \ref{def connectingmaps}. The \emph{maximal higher index} of $D$ is the element
			\begin{empheq}[left={\Ind_{\Gamma,\textnormal{max}}D\coloneqq
\empheqlbrace}]{alignat*=2}
    \partial_{1}\left[\chi(D)\right]\in K_{0}\big(C^*_\textnormal{max}(M)^{\Gamma}\big),\quad&\textnormal{ if $\dim M$ is even},\\[1.5ex]
    \partial_{0}\left[\tfrac{\chi(D)+1}{2}\right]\in K_{1}\big(C^*_\textnormal{max}(M)^{\Gamma}\big),\quad&\textnormal{ if $\dim M$ is odd}.
\end{empheq}
\end{definition}

\hfill\vskip 0.1in
\subsection{Functoriality}
\hfill\vskip 0.05in
\noindent In this subsection, we return to the geometric setup described in subsection \ref{subsec geom setup} and give a new proof of Theorem \ref{thm funct higher ind}. 

A key idea is to use the local nature of the wave operator to prove:
\begin{proposition}
	\label{prop wave}
	For all $t\in\mathbb{R}$, we have $\Psi(e^{itD_1})=e^{itD_2}$.
\end{proposition}
\begin{proof}
	For $j=1,2$, let $e^{itD_j}_{L^2}$ be the wave operator on $L^2(E_j)$. By \eqref{eq L2 multiplier}, this operator extends uniquely to a bounded multiplier of $C^*_{\textnormal{max}}(M_j)^{\Gamma_j}$. By Proposition \ref{prop waves equal}, we have $e^{itD_j}_{L^2}=e^{itD_j}\in\mathcal{M}_j$. Thus to prove this proposition, it suffices to show that $\Psi(e^{itD_1}_{L^2})=e^{itD_2}_{L^2}$, as elements of $B_{\textnormal{fp}}(L^2(E_j))^{\Gamma_j}$.

As a notational convenience, we will write $e^{itD_j}$ for $e^{itD_j}_{L^2}\in B_{\textnormal{fp}}(L^2(E_j))^{\Gamma_j}$.	
	Let the open covers $\mathcal{U}_{M_1}$ and $\mathcal{U}_{M_2}$ be as in \eqref{eq POU M1} and \eqref{eq POU M2}, so that by definition, there exists some $\epsilon>0$ such that each ball of diameter $\epsilon$ in $M_2$ is evenly covered with respect to $\pi\colon M_1\to M_2$. 
	
	Let $\{V_k\}$ be another open cover of $M_2$ such that each $V_k$ has diameter at most $\frac{\epsilon}{2}$ and such that any compact subset of $M_2$ intersects only finitely many of the $V_k$. Let $\{\rho_k\}$ be a partition of unity subordinate to $\{V_k\}$. 
	
	Choose a positive integer $n$ such that $\frac{t}{n}<\frac{\epsilon}{8}$. Now since
	$$e^{itD_1}=\big(e^{i\frac{t}{n}D_1}\big)^n,$$
	and $\Psi$ is a $*$-homomorphism, it suffices to show that $\Psi(e^{i\frac{t}{n}D_1})=e^{i\frac{t}{n}D_2}$. Noting that any section $u\in C_c(E_2)$ can be written as a finite sum $u=\sum_k\rho_k u,$ we have
	$$\Psi(e^{i\frac{t}{n}D_1})u=\sum_k\Psi(e^{i\frac{t}{n}D_1})(\rho_k u).$$
	We first claim that for each $k$ we have
	\begin{equation}
	\label{eq kth summand}
	\Psi(e^{i\frac{t}{n}D_1})(\rho_k u)=e^{i\frac{t}{n}D_2}(\rho_k u).
	\end{equation}
	To see this, note that the ball of radius $\frac{\epsilon}{8}$ around $\supp(\rho_k u)$ has diameter at most $\epsilon$ and so is evenly covered with respect to $\pi$. Since the definition of $\Psi$ is independent of the choice of compatible partitions of unity by Proposition \ref{prop folding}, we may work with a partition of unity $\{\phi_j^{[g]}\}$ for $M_2$ subordinate to a cover $\mathcal{U}_{M_2}$, such that $B_{\frac{\epsilon}{4}}(\supp(\rho_k u))\subseteq U_{j_0}$ for some open set $U_{j_0}^{[s_0]}\in\mathcal{U}_{M_2}$ and $\phi_{j_0}^{[s_0]}\equiv 1$ on $B_{\frac{\epsilon}{8}}(\supp(\rho_k u))$, for some $s_0\in S$ and $j_0\in I$. By the definition of the folding map \eqref{eq folding} applied to this choice of open cover, we have
	\begin{align*}
	\Psi(e^{i\frac{t}{n}D_1})(\rho_k u)&=\sum_{\substack{g\in\Gamma_1,\\s\in S}}\sum_{i,j\in I}\phi_i^{[g]}\big(e^{i\frac{t}{n}D_1}\phi_j^s(\rho_k u)\big)\\
	&=\phi_{j_0}^{[s_0]}\big(e^{i\frac{t}{n}D_1}\big(\pi|_{U_{j_0}^{s_0}}^*(\rho_k u)\big)\big)\\
	&=\pi_*\big(e^{i\frac{t}{n}D_1}\big(\pi|_{U_{j_0}^{s_0}}^*(\rho_k u)\big)\big).
	\end{align*}
	Now the wave equation on $M_1$ reads
	$$\frac{\partial}{\partial t}\big(e^{i\frac{t}{n}D_1}\big(\pi|_{U_{j_0}^{s_0}}^*(\rho_k u)\big)\big)=iD_1\big(e^{i\frac{t}{n}D_1}\big(\pi|_{U_{j_0}^{s_0}}^*(\rho_k u)\big)\big).$$
	Applying $\pi_*$ to both sides of this equation and using that $D_1$ is the lift of $D_2$, we obtain
	$$\frac{\partial}{\partial t}\big(\Psi\big(e^{i\frac{t}{n}D_1}\big)(\rho_k u)\big)=iD_2\big(\Psi\big(e^{i\frac{t}{n}D_1}\big)(\rho_k u)\big),$$
	whence \eqref{eq kth summand} follows from uniqueness of the solution to the wave equation on $M_2$. 
	
	Taking a sum over $k$ now yields
	$$\Psi(e^{i\frac{t}{n}D_1})u=\sum_k\Psi(e^{i\frac{t}{n}D_1})(\rho_k u)=\sum_k e^{i\frac{t}{n}D_2}(\rho_k u)=e^{i\frac{t}{n}D_2}u.$$
	Thus $\Psi(e^{i\frac{t}{n}D_1})=e^{i\frac{t}{n}D_2}$, as bounded operators on $L^2(E_2)$. By our previous remarks, this means that $\Psi(e^{it{D_1}})=e^{itD_2}$.
\end{proof}
Applying Fourier inversion together with Proposition \ref{prop wave} leads to:
\begin{proposition}
	\label{prop chi}
	For any $f\in C_0(\mathbb{R})$ we have 
	$$\Psi(f(D_1))=f(D_2)\in\mathcal{M}_2.$$
\end{proposition}
\begin{proof}
Suppose first that $f\in\mathcal{S}(\mathbb{R})$ with compactly supported Fourier transform. By the Fourier inversion formula, we have
	\begin{equation}
	\label{eq Fourier inv}
	f(D_j)=\frac{1}{2\pi}\int_{\mathbb{R}}\widehat{f}(t)e^{itD_j}\,dt,
	\end{equation} 
	where the integral converges strongly in $\mathcal{M}_j$. Now for any $\kappa\in C^*_\textnormal{max}(M_2)^{\Gamma_2}$, Proposition \ref{prop psi surj} implies that there exists 
	$\kappa_0\in C^*_\textnormal{max}(M_1)^{\Gamma_1}$ such that $\Psi(\kappa_0)=\kappa$. Since $\Psi$ is a $*$-homomorphism,
	\begin{align*}
	\Psi(f(D_1))(\kappa)&=\Psi(f(D_1)\kappa_0)
	=\frac{1}{2\pi}\int_\mathbb{R}\widehat{f}(t)\Psi(e^{itD_1}\kappa_0)dt.
	\end{align*}
	By Proposition \ref{prop wave}, this equals
	\begin{align*}
	\frac{1}{2\pi}\int_{\mathbb{R}}\widehat{f}(t)e^{itD_2}\kappa\,dt=f(D_2)\kappa.
	\end{align*}
	This proves the claim for $f\in\mathcal{S}(\mathbb{R})$. The general claim now follows from density of $\mathcal{S}(\mathbb{R})$ in $C_0(\mathbb{R})$.

%
\end{proof}

\begin{proof}[Proof of Theorem \ref{thm funct higher ind}]
	The expressions \eqref{eq even index} and \eqref{eq odd index} show that, for $j=1,2$, the higher index of $D_j$ is represented by a matrix $A_j(\chi)$ whose entries are operators formed using functional calculus of $D_j$. More precisely, if the initial operator $D_N$ on $N$ is ungraded, as is typically the case for $M$ odd-dimensional, then 
	\begin{equation}
	\label{eq Aj ungraded}
	A_j(\chi)=e^{\pi i(\chi+1)}(D_j).
	\end{equation}
	When $D_j$ is odd-graded with respect to a $\mathbb{Z}_2$-grading on the bundle $E_j=E_j^+\oplus E_j^-$, as typically occurs when $\dim N$ is even, we have a direct sum decomposition $\chi(D_j)=\chi(D_j)^+\oplus\chi(D_j)^-$. In this case, the index element is represented explicitly by the matrix
	\begin{equation}
	\label{eq Aj graded}
	A_j(\chi)=\left(\begin{smallmatrix}
			(1-\chi(D_j)^-\chi(D_j)^+)^2 & \,\,\chi(D_j)^-(1-\chi(D_j)^+\chi(D_j)^-)\\[1ex]
			\chi(D_j)^+(2-\chi(D_j)^-\chi(D_j)^+)(1-\chi(D_j)^-\chi(D_j)^+) & \,\,\chi(D_j)^+\chi(D_j)^-(2-\chi(D_j)^+\chi(D_j)^-)-1
		\end{smallmatrix}\right)	
	\end{equation}

	
	Observe that in either case, each entry of $A_j$ is an operator of the form $f(D_j)$, for some $f\in C_0(\mathbb{R})$ (modulo grading and the identity operator). By Proposition \ref{prop chi}, $\Psi$ maps each entry of $A_1$ to the corresponding entry of $A_2$. Hence $\Psi_*[A_1]=[A_2]$, which proves the claim.
\end{proof}
\begin{remark}
When $\Gamma_1=\pi_1 N$ and $\Gamma_2$ is the trivial group, Theorem \ref{thm funct higher ind} reduces to the maximal version of Atiyah's $L^2$-index theorem mentioned in section \ref{sec intro}.

Atiyah's original $L^2$-index theorem, which uses the von Neumann trace $\tau$ instead of the folding map, can be proved by an argument along lines similar to the proof of Theorem \ref{thm funct higher ind}.
\end{remark}
\hfill\vskip 0.2in
\section{Functoriality for the higher rho invariant}
\label{sec higher rho}
The higher index is a primary obstruction to the existence of positive scalar curvature metrics on a manifold. When the manifold is spin with positive scalar curvature, so that the higher index of the Dirac operator vanishes, one can define a secondary invariant called the \emph{higher rho invariant}, introduced in \cite{Roetopology,HR3}. This is an obstruction to the inverse of the Dirac operator being local \cite{Hongzhi}. In this section we show that the higher rho invariant behaves functorially under the map $\Psi_{L,0}$ from Definition \ref{def folding localization}. 

\subsection{Higher rho invariant}
\label{subsec higher rho}
\hfill\vskip 0.05in
\noindent Consider the geometric situation in subsection \ref{subsec geom setup}, with the additional condition that the Riemannian manifold $N$ is spin with positive scalar curvature. The operator $D_N$ is then the Dirac operator acting on the spinor bundle $E_N$.

As the definition of the higher rho invariant is the same for either $M_1$ or $M_2$, we will simply write $M$, $\Gamma$ to mean either $M_1$, $\Gamma_1$ or $M_2$, $\Gamma_2$. Similarly, $D$ will refer to either of the lifted operators $D_1$ or $D_2$ acting on the equivariant spinor bundles $E_1$ or $E_2$ lifted from $E_N$.

Let $\kappa$ be the scalar curvature function of the lifted metric on $M$, which is uniformly positive. Let $\nabla\colon C^\infty(E)\rightarrow C^\infty(T^*M\otimes E)$ be the connection on $E$ induced by the Levi-Civita connection on $M$. 
Recall that by the Lichnerowicz formula,
$$D^2=\nabla^*\nabla+\frac{\kappa}{4}.$$
Since $\kappa$ is uniformly positive, $D^2$ is strictly positive as an unbounded operator on the Hilbert module $C^*_{\textnormal{max}}(M)^{\Gamma}$. Thus we may use the functional calculus from Theorem \ref{thm:functionalcalculus} to form the operator
$$F_0(D)\coloneqq\frac{D}{|D|},$$
an element of the multiplier algebra $\mathcal{M}=\mathcal{M}(C^*_{\textnormal{max}}(M)^{\Gamma})$. Observe that $\frac{F_0(D)+1}{2}$ is a projection in $\mathcal{M}$. 

Since $D$ is invertible, there exists $\epsilon>0$ such that the spectrum of $D$ is contained in $\mathbb{R}\backslash(-\epsilon,\epsilon)$. Let $\{F_t\}_{t\in\mathbb{R}^+}$ be a set of normalizing functions satisfying the following conditions:
\begin{itemize}
\item $F_t$ has compactly supported distributional Fourier transform for each $t$;
\item $\textnormal{diam}(\textnormal{supp}\,\widehat{F}_t)\to 0$ as $t\to\infty$;
\item $F_t\to\frac{x}{|x|}$ uniformly on $\mathbb{R}\backslash(-\epsilon,\epsilon)$ in the limit $t\to 0$.
\end{itemize}
In the limit $t\to\infty$, the propagations of $F_t(D)$ tend to $0$. By Theorem 4.2 (i), as $t\to 0$, the operators $F_t(D)$ converge to $F_0(D)$ in the norm of $\mathcal{M}$.




Define a path $\mathbb{R}^{\geq 0}\rightarrow(C^*_{\textnormal{max}}(M)^{\Gamma})^+$ given by 
\begin{equation}
\label{eq RD}
R_{D}\colon t\mapsto A(F_t),
\end{equation}
where the matrix $A(F_t)$ is defined by \eqref{eq Aj ungraded} or \eqref{eq Aj graded} depending upon the dimension of $N$ (the subscript $j$ is omitted).
Noting that
$$R_{D}(0)=A(F_0)=
\begin{cases}
	\left(\begin{smallmatrix}0\,&\,0\\[1ex]0\,&\,\,1\end{smallmatrix}\right)&\textnormal{ if $\dim N$ is even},\\[1ex]
	\,\,\,\,\,\,\,1&\textnormal{ if $\dim N$ is odd},
\end{cases}
$$
one sees that $R_{D}$ is a matrix with entries in
$\big(C^*_{L,0,\textnormal{max}}(M)^{\Gamma}\big)^+$.
\begin{definition}
	\label{def higher rho}
	The \emph{higher rho invariant} of $D$ on the Riemannian manifold $M$ is
	$$\rho_{\textnormal{max}}(D)=\left[R_{D}\right]\in K_\bullet(C^*_{L,0,\textnormal{max}}(M)^{\Gamma}),$$ 
	where $\bullet=\dim M \textnormal{ (mod } 2)$.
\end{definition}
\hfill\vskip 0.01in
\subsection{Functoriality}
\hfill\vskip 0.05in
\noindent We are now ready to complete the proof of our main result, Theorem \ref{thm funct higher rho}, using the tools we developed in sections \ref{sec folding}, \ref{sec funct calc}, and \ref{sec higher ind}.

 	Recall from Definition \ref{def folding localization} that we have a folding map at level of obstructions algebras,
 	$$\Psi_{L,0}\colon C^*_{L,0,\textnormal{max}}(M_1)^{\Gamma_1}\to C^*_{L,0,\textnormal{max}}(M_2)^{\Gamma_2}.$$
 	This map is well-defined because the folding map $\Psi$ at the level of maximal equivariant Roe algebras preserves small propagation of operators, by Proposition \ref{prop prop}. The induced map on $K$-theory, 
 	$$(\Psi_{L,0})_*\colon K_{\bullet}\left(C^*_{L,0,\textnormal{max}}(M_1)^{\Gamma_1}\right)\rightarrow K_{\bullet}\left(C^*_{L,0,\textnormal{max}}(M_2)^{\Gamma_2}\right),$$
 	implements functoriality of the maximal higher rho invariant.
\begin{proof}[Proof of Theorem \ref{thm funct higher rho}]
	For $j=1,2$, let the higher rho invariants of $D_j$ be denoted by $\rho_{\textnormal{max}}(D_j)$, as in Definition \ref{def higher rho}. By \eqref{eq RD}, this class is represented the path
	$$R_{D_j}\colon t\mapsto A_j(F_t),$$
	where the matrix $A_j$ is as in \eqref{eq Aj ungraded} and \eqref{eq Aj graded}. By Definition \ref{def folding localization}, the map $(\Psi_L)_*$ takes the class 
	$$[R_{D_1}]\in K_\bullet\big(C^*_{L,0,\textnormal{max}}(M_1)^{\Gamma_1}\big)$$ to the class of the composed path 
	$$\Psi\circ R_{D_1}\colon t\mapsto\Psi(A_1(F_t))$$
	in $K_\bullet\big(C^*_{L,0,\textnormal{max}}(M_2)^{\Gamma_2}\big)$. Since each entry of $A_j$ is an operator of the form $f(D_j)$, for some $f\in C_0(\mathbb{R})$ (up to grading and the identity operator), Proposition \ref{prop chi} implies that for each $t\geq 0$, we have
	$$\Psi\circ R_{D_1}(t)=\Psi(A_1(F_t))=A_2(F_t)=R_{D_2}(t).$$
	It follows that	$(\Psi_{L,0})_*(\rho_{\textnormal{max}}(D_1))=\rho_{\textnormal{max}}(D_2)\in K_\bullet(C^*_{L,0,\textnormal{max}}(M_2)^{\Gamma_2})$.
\end{proof}

\hfill\vskip 0.2in
\section{Generalizations to the non-cocompact setting}
\label{sec generalizations}
The methods in this paper can be used to establish analogous results in more general geometric settings. In this final section, we give two such generalizations, both involving non-cocompact actions.

We will work with the non-cocompact analogue of the geometric setup in subsection \ref{subsec geom setup}, so that the manifold $N$ is no longer assumed to be compact. We will assume throughout this section that the operator $D_N$ has unit propagation speed. In place of the finite partition of unity \eqref{eq POU N} used to define the folding map $\Psi$, we take a locally finite partition of unity $\mathcal{U}_N$ whose elements are evenly covered with respect to the projections $p_1\colon M_1\to N$ and $p_2\colon M_2\to N$, and with the property that any compact subset of $N$ intersects only finitely many elements of $\mathcal{U}_N$. The equivariant partitions of unity $\mathcal{U}_{M_1}$ and $\mathcal{U}_{M_2}$ of $M_1$ and $M_2$ are defined in the same way according to \eqref{eq POU M1} and \eqref{eq POU M2}. 

The local nature of the wave operator $e^{itD}$ means that Proposition \ref{prop wave} generalizes naturally to this setting:

\begin{proposition}
	\label{prop wave non-cocompact}
	Let $N$, $M_1$, and $M_2$ be as in this section, with $N$ not necessarily compact. Then for all $t\in\mathbb{R}$, we have $\Psi(e^{itD_1})=e^{itD_2}$.
\end{proposition}
\begin{proof}
We adapt the proof of Proposition \ref{prop wave}, indicating only what needs to be changed. Let the open covers $\mathcal{U}_{M_1}$ and $\mathcal{U}_{M_2}$ be as above. The difference now is that since $N$ may be non-compact, we cannot assume the existence of a uniformly positive covering diameter $\epsilon$ as in the proof of Proposition \ref{prop wave}. 

Instead, the key point is to observe that for any \emph{fixed} $t\in\mathbb{R}$ and $u\in C_c(E_2)$, the section $e^{itD_2} u$ is supported within the compact subset $B_t(\supp u)$. Thus we can find $\epsilon>0$ such that for all $x\in B_t(\supp u)$, the ball $B_\epsilon(x)$ is evenly covered with respect to the projection $\pi\colon M_1\to M_2$. From here, we proceed precisely as in the proof of Proposition \ref{prop wave}, with $M_2$ replaced by $B_t(\supp u)$, to show that $\Psi(e^{itD_1})u=e^{itD_2}u$. Since $t$ and $u$ are arbitrary, we conclude.
\end{proof}

\subsection{Operators invertible at infinity}
\label{subsec inv infty}
\hfill\vskip 0.05in
\noindent 
In this subsection, suppose that the operator $D_N$ is \emph{invertible at infinity}, meaning that there exists some compact subset $Z_N\subseteq N$ on whose complement we have $D_N^2\geq a$ for some $a>0$. An important special case is when $N$ is spin, $D_N$ is the Dirac operator, and the metric $g_N$ has uniformly positive scalar curvature outside of $Z_N$. 

For $j=1,2$, the lifted operator $D_j$ then satisfies the analogous relation $D_j^2\geq a$ on the complement of a \emph{cocompact}, $\Gamma_j$-invariant subset $Z_j\subseteq M_j$. We can define a version of the higher index of $D_j$, localized around $Z_j$, as follows (see also \cite{Roecurvature} and \cite[section 3]{GHM3}). 

For each $R>0$, let $C^*_{\textnormal{max}}(B_R(Z_j))^{\Gamma_j}$ be the maximal equivariant Roe algebra of the $R$-neighborhood of $Z_j$. Since $B_R(Z_j)$ is cocompact, this algebra is isomorphic to $C^*_\textnormal{max}(\Gamma_j)\otimes\mathcal{K}$ by Remark \ref{rem max norm}. One can then show that for any $f\in C_c(-a,a)$, we have
	$$f(D_j)\in\lim_{R\rightarrow\infty}C^*_{\textnormal{max}}(B_R(Z_j))^{\Gamma_j},$$
	where $\lim_{R\rightarrow\infty}C^*_{\textnormal{max}}(B_R(Z_j))^{\Gamma_j}$ is the direct limit of these $C^*$-algebras. Indeed, this limit algebra is isomorphic to $C^*_\textnormal{max}(\Gamma)\otimes\mathcal{K}$, where $\mathcal{K}$ denotes the compact operators on a (not necessarily compact) fundamental domain of the $\Gamma$-action. The construction from subsection \ref{subsec higher ind}
	then gives an index element
		$$\Ind_{\Gamma_j,\textnormal{max}}D_j\in K_\bullet(C^*_{\textnormal{max}}(\Gamma_j)).$$
	
	Similar to the cocompact case, we have the following version of Theorem \ref{thm funct higher ind} for operators that are invertible at infinity:
	\begin{theorem}
	\label{thm generalized funct higher ind}
	Let $N$ be a Riemannian manifold and $D_N$ a first-order, self-adjoint elliptic differential operator acting on a bundle $E_N\rightarrow N$. Assume that $D_N$ has unit propagation speed. Let $M_1$ and $M_2$ be Galois covers of $N$ with deck transformation groups $\Gamma_1$ and $\Gamma_2\cong\Gamma_1/H$ respectively, for a normal subgroup $H$ of $\Gamma_1$. Let $D_1$ and $D_2$ be the lifts of $D_N$ to $M_1$ and $M_2$ respectively. Then the map on $K$-theory induced by the folding map $\Psi$ relates the maximal higher indices of $D_1$ and $D_2$:
	$$\Psi_*(\Ind_{\Gamma_1,\textnormal{max}}D_1)=\Ind_{\Gamma_2,\textnormal{max}}D_2\in K_\bullet\big(C^*_{\textnormal{max}}(\Gamma_2)\big).$$
	\end{theorem}
	\begin{proof}
		Analogous to the proof of Theorem \ref{thm funct higher ind}, with Proposition \ref{prop wave} replaced by Proposition \ref{prop wave non-cocompact}.
	\end{proof}

\hfill\vskip 0.1in
\subsection{Manifolds with equivariantly bounded geometry}
\hfill\vskip 0.05in
\noindent A second generalization involves non-cocompact coverings that satisfy certain additional geometric conditions. Under these conditions we obtain generalizations of both Theorems \ref{thm funct higher ind} and \ref{thm funct higher rho}. In contrast to Theorem \ref{thm generalized funct higher ind}, here we do not require invertibility of the operator at infinity, and the Roe algebra we work with does not need to be defined in terms of cocompact sets as in subsection \ref{subsec inv infty}.

Again, suppose that we are in the setup of subsection \ref{subsec geom setup}, but with $N$ not necessarily compact. We impose two conditions on the geometry of $N$ and the $\Gamma_i$-action on $M_i$ (see also \cite[subsection 2.1]{GXY}):
\begin{enumerate}[A.]
\item The Riemannian manifold $N$ has positive injectivity radius, and its curvature tensor is uniformly bounded across $N$ along with all of its derivatives.
\item For $j=1,2$, there exists a fundamental domain $\mathcal{D}_j$ for the action of $\Gamma_j$ on $M_j$ such that
$$l(\gamma)\rightarrow\infty\implies d(\mathcal{D}_j,\gamma\mathcal{D}_j)\to\infty,$$
where $l\colon\Gamma_j\to\mathbb{N}$ is a fixed length function and $d$ is the Riemannian distance on $M_j$.
\end{enumerate}

	It follows from these two assumptions and \cite[Proposition 2.14]{GXY} that the maximal equivariant Roe algebra $C^*_{\textnormal{max}}(M_j)^{\Gamma_j}$ is well-defined, along with a subalgebra $C^*_{\textnormal{max},u}(M_j)^{\Gamma_j}$ called the maximal equivariant \emph{uniform} Roe algebra. Using the latter, we constructed in \cite{GXY} a version of the functional calculus suitable for the maximal setting. 
	
	We briefly recall the definition of $C^*_{\textnormal{max},u}(M_j)^{\Gamma_j}$. Let $\mathcal{S}_u^{\Gamma_j}$ be the $*$-subalgebra of $\mathcal{S}^{\Gamma_j}$ (see section \ref{sec funct calc}) whose kernels have uniformly bounded  derivatives of all orders. In other words, an element of $\mathcal{S}_u^{\Gamma_j}$ is a bounded operator on $L^2(E_j)$ given by a Schwartz kernel $\kappa\in C_b^\infty(E_j\boxtimes E_j^*)$ such that
	\begin{enumerate}[(i)]
	\item $\kappa$ has finite propagation;
	\item $\kappa(x,y)=\kappa(\gamma x,\gamma y)$ for all $\gamma\in\Gamma_j$;
	\item Each covariant derivative of $\kappa$ is uniformly bounded over $M_j$.	
	\end{enumerate}
Then $C^*_{\textnormal{max},u}(M_j)^{\Gamma_j}$ is defined to be the closure of $\mathcal{S}_u^{\Gamma_j}$ in $C^*_{\textnormal{max}}(M_j)^{\Gamma_j}$. 

Similar to section \ref{sec funct calc}, $C^*_{\textnormal{max},u}(M_j)^{\Gamma_j}$ can be viewed as right Hilbert module over itself, with inner product and multiplication being defined by the same formula \eqref{eq Hilbert module structure}. The operator $D_j$ can be viewed as a densely defined symmetric operator on $C^*_{\textnormal{max},u}(M_j)^{\Gamma_j}$. By \cite[Theorem 3.1]{GXY}, this operator is regular and so admits a functional calculus. The construction from subsection \ref{subsec higher ind}
	then allows one to define the \emph{maximal equivariant uniform index} of $D_j$,
		$$\ind_{\Gamma_j,\textnormal{max},u}D_j\in K_\bullet(C^*_{\textnormal{max},u}(M_j)^{\Gamma_j}).$$
	The corresponding versions of the higher rho invariant can be defined as in subsection \ref{subsec higher rho}, and we denote this by
	$$\rho_{\textnormal{max,u}}(D_j)\in K_\bullet(C^*_{L,0,\textnormal{max,u}}(M_j)^{\Gamma_j}),$$
	where  $C^*_{L,0,\textnormal{max,u}}(M_j)^{\Gamma_j}$ is defined analogously to $C^*_{L,0,\textnormal{max}}(M_j)^{\Gamma_j}$, with $C^*_{\textnormal{max},u}(M_j)^{\Gamma_j}$ in place of $C^*_{\textnormal{max}}(M_j)^{\Gamma_j}$. Constructions analogous to those in subsection \ref{subsec induced} give way to folding maps at $K$-theory level, which we denote by
	$$(\Psi_u)_*\colon K_{\bullet}(C^*_{\textnormal{max},u}(M_1)^{\Gamma_1})\rightarrow K_{\bullet}(C^*_{\textnormal{max},u}(M_2)^{\Gamma_2}),$$
$$(\Psi_{L,0,u})_*\colon K_{\bullet}(C^*_{L,0,\textnormal{max},u}(M_1)^{\Gamma_1})\rightarrow K_{\bullet}(C^*_{L,0,\textnormal{max},u}(M_2)^{\Gamma_2}).$$
	
Further, Proposition \ref{prop waves equal} generalizes naturally to manifolds satisfying conditions A and B above:
\begin{proposition}
\label{prop waves equal noncocompact}
For each $t\in\mathbb{R}$, we have
$$e^{itD_j}_{L^2}=e^{itD_j}\in\mathcal{M}_j(C^*_{\textnormal{max},u}(M_j)^{\Gamma_j}),$$
with notation as in subsection \ref{subsec wave}.
\end{proposition}
\begin{proof}
The proof proceeds exactly as for Proposition \ref{prop waves equal}, but with Lemma \ref{lem pointwise wave} replaced by Lemma \ref{lem norm cts noncocompact} below.	
\end{proof}
\begin{lemma}
\label{lem norm cts noncocompact}
Let $N$, $M_1$, and $M_2$ be as in this section, with $N$ not necessarily compact. Let $\kappa\in\mathcal{S}_u^{\Gamma_j}$, and let $e^{itD_j}_{L^2}\kappa$ denote the smooth Schwartz kernel of $e^{itD_j}_{L^2}\circ T_\kappa$. Then the path $t\mapsto e^{itD_j}_{L^2}\kappa$ is continuous with respect to the operator norm, and we have
\begin{equation}
\label{eq wave kernel}
\lim_{h\to 0}\,\,\,\bigg\|\frac{e^{i(t+h)D_j}_{L^2}\circ T_\kappa-e^{itD_j}_{L^2}\circ T_\kappa}{h}-iD_j e^{itD_j}_{L^2}\circ T_\kappa\bigg\|_{\mathcal{B}(L^2(E))}= 0.
\end{equation}
\end{lemma}
\begin{proof}
	Similar to equation \eqref{eq mean value}, we have, for each $x,y\in M_j$ and $t\in[-t_0,t_0]$, 
\begin{equation*}
\Big|e^{itD_j}_{L^2}\kappa(x,y)-\kappa(x,y)\Big|\leq |t|\cdot\sup_{s\in[-t_0,t_0]}\Big|iD_je^{isD_j}_{L^2}\kappa(x,y)\Big|\,.
\end{equation*}
Since $M_j$ has bounded Riemannian geometry, it follows from Sobolev theory that the operator $iD_je^{isD_j}_{L^2}\kappa(x,y)$ has uniformly bounded smooth kernel, along with all derivatives. Indeed, since $e^{isD_j}_{L^2}$ is unitary, $\big|iD_je^{isD_j}_{L^2}\kappa(x,y)\big|$ is bounded above by a constant $C_1$ independent of $s$, $x$, and $y$. It follows that we can estimate the $L^2$-norm of the finite-propagation operator $e^{itD_j}_{L^2}\kappa-\kappa$ by
$$\norm{e^{itD_j}_{L^2}\kappa-\kappa}_{\mathcal{B}(L^2(E))}\leq C_2\cdot\sup_{x,y\in M}\Big|e^{itD_j}_{L^2}\kappa(x,y)-\kappa(x,y)\Big|\leq C_1C_2,$$
for some other constant $C_2$. We can then finish the proof as in Lemma \ref{lem pointwise wave}.
\end{proof}

With this in hand, we arrive (via an obvious analogue of Proposition \ref{prop chi}) at the following generalizations of Theorems \ref{thm funct higher ind} and \ref{thm funct higher rho}.

\begin{theorem}
	\label{thm more generalized funct higher ind}
	Let $N$ be a Riemannian manifold and $M_1$ and $M_2$ be Galois covers of $N$ with deck transformation groups $\Gamma_1$ and $\Gamma_2\cong\Gamma_1/H$ respectively, for a normal subgroup $H$ of $\Gamma_1$. Suppose that the conditions \textnormal{A} and \textnormal{B} in this subsection are satisfied. Let $D_N$ be a first-order, self-adjoint elliptic differential operator acting on a bundle $E_N\rightarrow N$, and assume that $D_N$ has unit propagation speed. Let $D_1$ and $D_2$ be the lifts of $D_N$ to $M_1$ and $M_2$ respectively. We then have:
	\begin{align*}
	(\Psi_{u,*})(\ind_{\Gamma_1,\textnormal{max,u}}D_1)&=\ind_{\Gamma_2,\textnormal{max,u}}D_2\in K_\bullet\big(C^*_{\textnormal{max},u}(M_2)^{\Gamma_2}\big),\\
	(\Psi_{u,*})(\ind_{\Gamma_1,\textnormal{max}}D_1)&=\ind_{\Gamma_2,\textnormal{max}}D_2\in K_\bullet\big(C^*_{\textnormal{max}}(M_2)^{\Gamma_2}\big).
	\end{align*}
\end{theorem}

\begin{theorem}
	\label{thm more generalized funct higher rho}
	Suppose $(N,g_N)$ is a spin Riemannian manifold with uniformly positive scalar curvature. Let $M_1$ and $M_2$ be Galois covers of $M$ with deck transformation groups $\Gamma_1$ and $\Gamma_2\cong\Gamma_1/H$ respectively, for a normal subgroup $H$ of $\Gamma_1$. Suppose that the conditions \textnormal{A} and \textnormal{B} in this subsection are satisfied. Let $D_N$ be the Dirac operator on $N$ (with unit propagation speed). Let $D_1$ and $D_2$ be the lifts of $D_N$ to $M_1$ and $M_2$ respectively. We then have:
	\begin{align*}
		(\Psi_{L,0,u})_*\big(\rho_{\textnormal{max,u}}(D_1)\big)=\rho_{\textnormal{max,u}}(D_2)\in K_\bullet\big(C^*_{L,0,\textnormal{max,u}}(M_2)^{\Gamma_2}\big),\\
		(\Psi_{L,0,u})_*\big(\rho_{\textnormal{max}}(D_1)\big)=\rho_{\textnormal{max}}(D_2)\in K_\bullet\big(C^*_{L,0,\textnormal{max}}(M_2)^{\Gamma_2}\big).
	\end{align*}
\end{theorem}
\begin{remark}
The second equality in Theorem \ref{thm more generalized funct higher ind} follows from the observation that the natural inclusion $C^*_{\textnormal{max},u}(M_j)^{\Gamma_j}\hookrightarrow C^*_{\textnormal{max}}(M_j)^{\Gamma_j}$ relates $\ind_{\Gamma_j,\textnormal{max,u}}D_j$ to $\ind_{\Gamma_j,\textnormal{max}}D_j$. The second equality in Theorem \ref{thm more generalized funct higher rho} follows from a similar relationship at the level of obstruction algebras.
\end{remark}

\hfill\vskip 0.2in

\hfill\vskip 0.2in
%
%
%
\bibliographystyle{plain}
\bibliography{../../BigBibliography/mybib}

\end{document}